\documentclass{amsart}
\usepackage{verbatim,amsmath,amssymb,amscd,color,graphics}
\usepackage{tikz,tabmac}
\usepackage[all]{xy}
 
\usepackage{comment}
 
\usepackage[colorlinks=true, pdfstartview=FitV, linkcolor=blue, citecolor=blue, urlcolor=blue]{hyperref}

\newcommand{\arxiv}[1]{\href{http://arxiv.org/abs/#1}{\tt arXiv:\nolinkurl{#1}}}

\newcommand{\bz}{{\Bbb Z}}
 
\newcommand{\aff}{\mathrm{aff}}
\newcommand{\ch}{\mathrm{ch}}
\newcommand{\intrinsic}{\mathrm{int}}
\newcommand{\g}{\mathfrak{g}}
\newcommand{\lev}{\mathrm{lev}}
\newcommand{\wt}{\mathrm{wt}\,}
\newcommand{\oP}{\overline{P}}
\newcommand{\oW}{\overline{W}}
\newcommand{\Q}{\mathbb{Q}}
\newcommand{\Z}{\mathbb{Z}}

\newcommand\tab[1]{\begin{array}{|c|}\hline {\lower1pt\hbox{$#1$}}\\ \hline \end{array}}
\newcommand{\La}{\Lambda}
\newcommand{\la}{\lambda}
\newcommand{\os}{\text{outer}}
\newcommand{\is}{\text{inner}}

\newcommand{\ol}{\overline}

\newcommand{\Aut}{\text{Aut}}
\newcommand{\ms}{\text{int}}
  
\numberwithin{equation}{section}
 
\newtheorem{theorem}{Theorem}
\newtheorem{prop}[theorem]{Proposition}
\newtheorem{lemma}[theorem]{Lemma}
\newtheorem{corollary}[theorem]{Corollary}

\theoremstyle{definition}

\newtheorem{definition}[theorem]{Definition}
\newtheorem{remark}[theorem]{Remark}
\newtheorem{example}[theorem]{Example}
 
\numberwithin{theorem}{section}

\begin{document}
 
\title[Demazure crystals and energy function]
{Demazure crystals, Kirillov--Reshetikhin crystals, and the energy function}
 
\author[A.~Schilling]{Anne Schilling}
\address{Department of Mathematics, University of California, One Shields
Avenue, Davis, CA 95616-8633, U.S.A.}
\email{anne@math.ucdavis.edu}
\urladdr{http://www.math.ucdavis.edu/\~{}anne}
 
\author[P.~Tingley]{Peter Tingley}
\address{Department of Mathematics, M.I.T., Cambridge, MA 02139-4307, U.S.A.}
\email{peter.tingley@gmail.com}
\urladdr{http://www-math.mit.edu/\~{}ptingley/}
 

\begin{abstract}
It has previously been shown that, at least for non-exceptional Kac--Moody Lie algebras,
there is a close connection between Demazure crystals and tensor products of
Kirillov--Reshetikhin crystals. In particular, certain Demazure crystals are isomorphic as
classical crystals to tensor products of  Kirillov--Reshetikhin crystals via a canonically chosen
isomorphism. Here we show that this isomorphism intertwines the natural affine grading on
Demazure crystals with a combinatorially defined energy function.
As a consequence, we obtain a formula of the Demazure character in terms of the energy function,
which has applications to Macdonald polynomials and $q$-deformed Whittaker functions.
\end{abstract}
 
\maketitle

\tableofcontents
 
\section{Introduction}
 
Kashiwara's theory of crystal bases~\cite{Kashiwara:1994} provides a remarkable combinatorial
tool for studying highest weight representations of symmetrizable Kac--Moody algebras and their 
quantizations. Here we consider finite-dimensional representations of the quantized universal enveloping 
algebra $U_q'(\g)$ corresponding to the derived algebra $\g'$ of an affine Kac--Moody algebra. These 
representations do not extend to representations of $U_q(\g)$, but one can nonetheless
define the notion of a crystal basis. In this setting crystal bases do not always exist, but
there is an important class of finite-dimensional modules for $U_q'(\g)$ that are known to
admit crystal bases: tensor products of the Kirillov--Reshetikhin
modules $W^{r, s}$ from~\cite{KR:1987} (denoted $W(s \omega_r)$ in that paper), where
$r$ is a node in the classical Dynkin diagram and $s$ is a positive integer.
 
The modules $W^{r, s}$ were first conjectured to admit crystal bases $B^{r,s}$
in~\cite[Conjecture 2.1]{HKOTY:1999}, and moreover it was conjectured that these crystals
are perfect whenever $s$ is a multiple of a particular constant $c_r$ (perfectness is a technical
condition which allows one to use the finite crystal to construct highest weight crystals,
see~\cite{KMN1:1992}). This conjecture has now been proven in all non-exceptional cases
(see~\cite{Okado:2007,OS:2008} for a proof that the crystals exist,
and~\cite[Theorem 1.2]{FOS:2010} for a proof that they are perfect). We call $B^{r,s}$ a Kirillov--Reshetikhin 
(KR) crystal.
 
The perfectness of KR crystals ensures that they are related to highest weight
affine crystals via the construction in~\cite{KMN1:1992}. In~\cite{Kashiwara:2005}, Kashiwara
proposed that this relationship is connected to the theory of Demazure
crystals~\cite{Kashiwara:1993,Littelmann:1995}, by conjecturing that perfect KR crystals are
isomorphic as classical crystals to certain Demazure crystals (which are subcrystals of affine
highest weight crystals). This was proven in most cases in~\cite{FL:2006,FL:2007}. More relations between 
Demazure crystals and tensor products of perfect KR crystals
were investigated in~\cite{KMOU:1998, KMOTU:1998,KMOTU:1998a,FSS:2007,Naoi:2011}.
 
There is a natural grading $\deg$ on a highest weight affine crystal $B(\Lambda)$, where
$\deg(b)$ records the number of $f_0$ in a string of $f_i$'s that act on the highest weight
element to give $b$ (this is well-defined by weight considerations). Due to the ideas
discussed above, it seems natural that this grading should transfer to a grading
on a tensor product of KR crystals.
 
Gradings on tensor products of KR crystals have in fact been studied,
and are usually referred to as ``energy functions."  They are vast generalizations of the major index
statistics on words, which can be viewed as elements in $(B^{1,1})^{\otimes L}$.
The idea dates to the earliest works on perfect crystals~\cite{KMN1:1992, KMN2:1992}, and was expanded 
in~\cite{OSS:2003} following conjectural definitions in~\cite{HKOTT:2002}. A function $D$, which we will refer 
to as the $D$-function,
is defined as a sum involving local energy functions for each pair of factors in the tensor product
and an `intrinsic energy' of each
factor. It has been suggested that there is a simple global characterization
of intrinsic energy related to the affine grading on a corresponding highest weight crystal
(see~\cite[Section 2.5]{OSS:2002}, \cite[Proof of Proposition 3.9]{HKOTT:2002}). However,
as far as we know, it has previously never been shown that the explicit definition of
intrinsic energy actually satisfies this condition.
 
\subsection{Results}
In the present work, we restrict to non-exceptional type (i.e. all affine Kac--Moody algebras
except $A_2^{(2)}$, $G_2^{(1)}$, $F_4^{(1)}$, $E_6^{(1)}$, $E_7^{(1)}$, $E_8^{(1)}$,
$E_6^{(2)}$ and $D_4^{(3)})$, where KR crystals are known to exist. The relationship between KR crystals and 
Demazure crystals discussed above can be enhanced: There is in fact a unique isomorphism of classical crystals 
between any tensor product $B = B_1 \otimes \cdots \otimes B_k$ of perfect level $\ell$ KR crystals and a certain 
Demazure crystal such that all $0$ arrows on the Demazure side correspond to $0$ arrows on the KR sides 
(although there are more $0$ arrows in the KR crystal). See Theorem \ref{FSS-theorem} for the precise statement.
This was proven by Fourier, Shimozono and the first author in~\cite[Theorem 4.4]{FSS:2007} under certain assumptions
since at the time KR crystals were not yet known to exist. Here we point out that in most cases the assumptions 
from~\cite{FSS:2007} follow from~\cite{Okado:2007,OS:2008,FOS:2009,FOS:2010}. In the remaining cases of
type $A_{2n}^{(2)}$ and the spin nodes for type $D_n^{(1)}$ we prove the statement separately,
thereby firmly establishing this relationship between KR crystals and Demazure crystals in all non-exceptional types.   

We then establish the correspondence between the $D$ function and the affine grading discussed earlier in this introduction. 
That is, we show that the unique isomorphism discussed above intertwines the 
basic grading on the Demazure crystal with the $D$-function on the KR crystal, up to a shift (i.e. addition of a global constant). 
This also allows for a new characterization of the $D$ function. Define the intrinsic energy function $E^{\intrinsic}$ on $B$ by 
letting $E^{\intrinsic}(b)$ record the minimal
number of $f_0$ in a path from a certain fixed $u \in B$ to $b$. We show that $E^{\intrinsic}$ agrees with the $D$-function 
up to a shift (i.e. addition of a global constant). The isomorphism between KR crystals and Demazure crystals actually 
intertwines the basic grading with $E^\intrinsic$ exactly. In particular this shows that for any Demazure crystal arrow
(see Remark~\ref{remark.demazure arrow}) the energy changes by 1 on the corresponding classical crystal
components.
 
We also consider the more general setting when $B$ is a tensor product of KR crystals which
are not assumed to be perfect or of the same level. The $D$ function is still well-defined, and we give a
precise relationship between $D$ and the affine grading on a related direct sum of highest weight modules
in Corollaries~\ref{cor:gen1} and~\ref{cor:gen2}.
In this case we no longer give an interpretation in terms of Demazure modules, although for tensor
products of perfect crystals of various levels it was subsequently shown by Naoi~\cite{Naoi:2011} that
the result is the disjoint union of Demazure crystals; see also Remark~\ref{general-demazure}.

\subsection{Applications} 
Our results express the characters of certain Demazure modules
in terms of the intrinsic energy on a related tensor product of KR crystals
(see Corollary~\ref{Echar}). This has potential applications whenever those Demazure characters appear.
 
For untwisted simply-laced root systems, Ion~\cite{Ion:2003}, generalizing results of Sanderson~\cite{San:2000} 
in type $A$, showed that the specializations $E_\lambda(q,0)$ of nonsymmetric 
Macdonald polynomials at $t=0$ coincide with specializations of Demazure characters of level one affine 
integrable modules. If $\lambda$ is anti-dominant, then $E_\lambda(q,0)$ is actually a 
symmetric Macdonald polynomial $P_\lambda(q,0)$. In this case, the relevant Demazure module is 
associated to a tensor product $B$ of level one KR crystals as above, so our results imply that $P_\lambda(q,0)$ 
is the character of $B$, where the powers of $q$ are given by $-D$, see Corollary~\ref{Mac=KR}.
It follows that the coefficients in the expansion of $P_\lambda(q,0)$ in terms of the irreducible characters
are the one-dimensional configuration sums defined in terms of the intrinsic energy in~\cite{HKOTT:2002}. 
 
There is also a relation between Demazure characters and $q$-deformed Whittaker
functions for $\mathfrak{gl}_n$~\cite[Theorem 3.2]{GLO:2008}. Hence our results allow one to study 
Whittaker functions via KR crystals.
 
\subsection{Exceptional types}
KR crystals are still expected to exist in exceptional types, although in most cases this has not yet been
established (see for example~\cite{JS:2010, Kashiwara:2002, KMOY:2007, Yamane:1997}
for some cases where it has).  Furthermore, it is expected that the relationship between KR crystals and Demazure
crystals holds in general (see~\cite{Kashiwara:2005} for the conjecture,
and~\cite{FL:2006, FL:2007} for a proof that it holds in some exceptional cases). Our expectation is
that Theorem~\ref{FSS-theorem} and Corollary~\ref{Edeg} would also continue to hold in all cases.
 
\subsection{Outline}
In Section~\ref{section.crystals} we briefly review the theory of crystals. In Section~\ref{section.finite.type}
we review combinatorial models for non-exceptional finite type crystals and their branching rules.
Section~\ref{section:realization} is devoted to combinatorial models for KR crystals.
In Section~\ref{section.energy}, we introduce the two energy functions $E^\intrinsic$ and
$D$. Section~\ref{section.demazure} discusses the isomorphism between Demazure crystals and tensor
products of perfect KR crystals. In Section~\ref{section.energy.equiv} we relate the basic grading on a Demazure
crystal with the energy function on a tensor product of KR crystals and show that the two energy functions agree
up to a shift (Theorems~\ref{grading=energy} and~\ref{E=D}).
Some of these results are generalized to tensor products of (not necessarily perfect) KR crystals of different level
in Section~\ref{section.nonperfect}.
In Section~\ref{section.applications} we discuss applications to Demazure characters,
nonsymmetric Macdonald polynomials, and Whittaker functions.
 
\subsection*{Acknowledgments}
We would like to thank Dan Bump, Ghislain Fourier, Stavros Kousidis, Cristian Lenart, Sergey Oblezin, 
and Masato Okado for enlightening discussions, and Nicolas Thi\'ery for his help with {\tt Sage}. Most of the
KR crystals have been implemented in the open-source mathematics system {\tt Sage}
({\tt sagemath.org}) by the first author. A thematic tutorial on crystals and affine crystals written
by Dan Bump and the first author is available in {\tt Sage}. A copy can be found at\newline
{\tt http://www.math.ucdavis.edu/\~{}anne/sage/lie/crystals.html}\newline
{\tt http://www.math.ucdavis.edu/\~{}anne/sage/lie/affine\_crystals.html}.\newline
We would like to thank the Fields Institute in Toronto and the Hausdorff Institut in Bonn for their hospitality, 
where part of this work was done.
AS was in part supported by NSF grants DMS--0652641, DMS--0652652, and DMS--1001256.
PT was partially supported by NSF grant DMS-0902649.

\section{Kac--Moody algebras and Crystals}
\label{section.crystals}
 
\subsection{General setup}
Let $\g$ be a Kac--Moody algebra. Let $\Gamma=(I,E)$ be its Dynkin diagram, where $I$ is the
set of vertices and $E$ the set of edges. Let $\Delta$ be the corresponding root system and  
$\{\alpha_i \mid i\in I\}$ the set of simple roots. Let $P,Q, P^\vee$, and $Q^\vee$ denote the weight lattice, 
root lattice, coweight lattice, and coroot lattice respectively. 
 
Let $U_q(\g)$ be the corresponding quantum enveloping algebra over ${\Bbb Q}(q)$.
Let $\{ E_i, F_i\}_{i\in I}$ be the standard elements in $U_q(\g)$ corresponding to the
Chevalley generators of the derived algebra $\g'$. We recall the triangular decomposition
\[
	U_q(\g) \cong  U_q(\g)^{<0} \otimes U_q(\g)^0 \otimes U_q(\g)^{>0},  
\]
where $U_q(\g)^{<0}$ is the subalgebra generated by the $F_i$, $U_q(\g)^{>0}$ is the
subalgebra generated by the $E_i$, $U_q(\g)^0$ is the abelian group algebra generated
by the usual elements $K_w$ for $w \in P^\vee$, and the isomorphism is as vector spaces.
Let $U_q'(\g)$ be the subalgebra generated by $E_i, F_i$ and $K_i := K_{\alpha_i^\vee}$ for $i \in I$.
 
We are particularly interested in the case when $\g$ is of affine type. In that situation, we use the following 
conventions:  $\Lambda_i$ to denotes the fundamental
weight corresponding to $i \in I$. For $i \in I \backslash \{ 0 \}$, $\omega_i$ denotes
the fundamental weight corresponding to that node in the related finite type Lie algebra.
The corresponding finite type weight lattice and Weyl group are denoted $\oP$ and $\oW$, respectively.
 
\subsection{$U_q(\g)$ crystals}  \label{Ug-cryst}
Here we give a very quick review, and refer the reader to \cite{HK:2002} more details. For us, a crystal
is a nonempty set $B$ along with operators $e_i : B \rightarrow B \cup\{ 0 \}$ and
$f_i : B \rightarrow B \cup \{ 0 \}$ for $i\in I$, which satisfy some conditions. The set $B$
records certain combinatorial data associated to a representation $V$ of $U_q(\g),$ and the operators $e_i$ 
and $f_i$ correspond to the Chevalley
generators $E_i$ and $F_i$.  The relationship between the crystal $B$ and the module
$V$ can be made precise using the notion of a crystal basis for $V$.
 
Often the definition of a crystal includes three functions
$\wt, \varphi, \varepsilon: B \rightarrow P$, where $P$ is the weight lattice. In the case of
crystals of integrable modules, these functions can be recovered (up to a  global shift in a
null direction in cases where the Cartan matrix is not invertible) from the knowledge of $e_i$
and $f_i$, as we discuss in a slightly different context in Section \ref{uprime-crystals}.
 
An important theorem of Kashiwara shows that every integrable $U_q(\g)$ highest weight
module $V(\lambda)$ has a crystal basis and hence a corresponding crystal $B(\lambda)$.
 
\subsection{$U_q'(\g)$ crystals} \label{uprime-crystals}
In the case when the Cartan matrix is not invertible, one can define an extended notion
of $U_q'(\g)$ crystal bases and crystals that includes some cases which do not lift to $U_q(\g)$ crystals. 
Note however that not all integrable $U'_q(\g)$ representations have corresponding crystals. 
See e.g.~\cite{Kashiwara:2002}. We consider only integrable crystals $B$, i.e. crystals where each $e_i, f_i$ 
acts locally nilpotently. Define a weight function on such a $B$ as follows:
First set
\begin{align*}
\varepsilon_i (b) & := \max \{ m \mid e_i^m (b) \neq 0 \}, \\
\varphi_i (b) & := \max \{ m \mid f_i^m (b) \neq 0 \}.
\end{align*}
Let
$\Lambda_i$ be the fundamental weight associated to $i \in I$.  For each $b \in B$, define
\begin{enumerate}
\item $\displaystyle \varphi(b):= \sum_{i \in I} \varphi_i (b) \Lambda_i$,
\item $\displaystyle \varepsilon(b) := \sum_{i \in I} \varepsilon_i (b) \Lambda_i$,
\item $\displaystyle \wt (b):=\varphi(b)-\varepsilon(b)$.
\end{enumerate}
Then $\wt(b)$ corresponds to the classical weight grading of the $U_q(\g)$ module associated to $B$, and is referred 
to as the weight function. Notice that $\wt(b)$ is always in
\begin{equation}
P':=  \text{span} \{ \Lambda_i : i \in I \}.
\end{equation}
If the Cartan matrix of $\g$ is not invertible, $P'$ is a proper sublattice of $P$.
 
\begin{remark}
One would expect that $f_i$ should have weight $-\alpha_i$. This is true for $U_q(\g)$ crystals, but for 
$U'_q(\g)$ crystals the weight of $f_i$ is actually the projection of
$-\alpha_i$ onto $P'$ under the projection that sends null roots to $0$.
\end{remark}
 
\subsection{Tensor products of crystals}

The tensor product rule for $U'_q(\g)$ or $U_q(\g)$ modules leads to a tensor product rule for the corresponding 
crystals. Following~\cite{FSS:2007}, we use the opposite conventions from Kashiwara~\cite{Kashiwara:1994}.
If $A$ and $B$ are two crystals, the tensor product $A \otimes B$ is the crystal whose underlying
set is the Cartesian set $\{ a \otimes b\mid a \in A, b \in B \}$ with operators $e_i$ and $f_i$
defined by:
\begin{equation}
\begin{aligned}
& e_i (a \otimes b)=
\begin{cases}
e_i  (a) \otimes b  & \text{if $\varepsilon_i(a) > \varphi_i(b)$,}\\
a \otimes e_i  (b)  & \text{otherwise,}
\end{cases} \\
& f_i (a \otimes b)=
\begin{cases}
f_i  (a) \otimes b & \text{if $\varepsilon _i(a) \geq \varphi_i(b)$,}\\
a \otimes f_i (b)  & \text{otherwise.}
\end{cases}
\end{aligned}
\end{equation}
   
\subsection{Abstract crystals}
 
\begin{definition} \label{abstract-C}
Let $\g$ be a Kac--Moody algebra with Dynkin diagram $\Gamma= (I, E).$ Let $B$ be a
nonempty set with operators $e_i, f_i : B \rightarrow B \sqcup \{ 0 \}$. We say $B$ is a
{\it regular abstract crystal} of type $\g$ if, for each pair $i \neq j \in I$, $B$ along with
the operators $e_i, f_i, e_j, f_j$ is a union of (possibly infinitely many) integrable highest weight 
$U_q(\g_{i,j})$ crystals, where $\g_{i,j}$ is the Lie algebra with
Dynkin diagram containing $i$ and $j$, and all edges between them. 
\end{definition}
 
\begin{remark} If $\g$ is an affine algebra other than $\widehat{\mathfrak{sl}}_2$,
then any $U_q'(\g)$ crystal is a regular abstract crystal of type $\g$.
\end{remark}
 
\subsection{Perfect crystals}

In this section, $\g$ is of affine type. 
Let $c=\sum_{i\in I} a_i^\vee \alpha_i^\vee$ be the canonical central element associated to $\g$ and $P^+$
the set of dominant weights, that is, $P^+ = \{ \Lambda \in P \mid \Lambda(\alpha_i^\vee) \in \mathbb{Z}_{\ge 0} \}$.
We define the {\it level} of $\Lambda \in P^+$ by $\lev(\Lambda):= \Lambda(c)$. For each $\ell \in \bz$, we 
consider the sets
\begin{equation*}
        P_\ell = \{ \Lambda \in P \mid \lev(\Lambda) = \ell \} \quad \text{and} \quad P_\ell^+
        = \{ \Lambda \in P^+ \mid \lev(\Lambda) = \ell \},
\end{equation*}
the sets of level-$\ell$ weights and level-$\ell$ dominant weights respectively. Note that $P^+_\ell = \emptyset$
if $\ell < 0$. The following important notion was introduced in \cite{KMN1:1992}.
 
\begin{definition} \label{def:perfect}
For a positive integer $\ell > 0$, a $U'_q(\g)$-crystal $B$ is called a
{\it perfect crystal of level $\ell$} if the
following conditions are satisfied:
\begin{enumerate}
\item $B$ is isomorphic to the crystal graph of a finite-dimensional
$U_q'(\mathfrak{g})$-module.
\item
$B\otimes B$ is connected.
\item
There exists a $\lambda\in \overline{P}$, such that
$\wt(B) \subset \lambda + \sum_{i\in I \setminus \{0\}} \mathbb{Z}_{\le 0} \alpha_i$ and there
is a unique element in $B$ of classical weight $\lambda$.
\item
$\forall \; b \in B, \;\; \lev(\varepsilon (b)) \geq \ell$.
\item \label{perfect:bij}
$\forall \; \Lambda \in P_\ell^{+}$, there exist unique elements
$b_{\Lambda}, b^{\Lambda} \in B$, such that
$$ \varepsilon ( b_{\Lambda}) = \Lambda = \varphi( b^{\Lambda}). $$
\end{enumerate}
\end{definition}
Examples of a perfect and nonperfect crystal is given in Figure~\ref{figure:KR example}.
 
 \begin{figure}
  \begin{center}
 \begin{tabular}{cc}
 \scalebox{.8}{
 \hspace{0.5cm}
 \begin{tikzpicture}[>=latex,line join=bevel,]
\node (1) at (11bp,238bp) [draw,draw=none] {${\def\lr#1{\multicolumn{1}{|@{\hspace{.6ex}}c@{\hspace{.6ex}}|}{\raisebox{-.3ex}{$#1$}}}\raisebox{-.6ex}{$\begin{array}[b]{c}\cline{1-1}\lr{1}\\\cline{1-1}\end{array}$}}$};
  \node (2) at (32bp,162bp) [draw,draw=none] {${\def\lr#1{\multicolumn{1}{|@{\hspace{.6ex}}c@{\hspace{.6ex}}|}{\raisebox{-.3ex}{$#1$}}}\raisebox{-.6ex}{$\begin{array}[b]{c}\cline{1-1}\lr{2}\\\cline{1-1}\end{array}$}}$};
  \node (-1) at (12bp,10bp) [draw,draw=none] {${\def\lr#1{\multicolumn{1}{|@{\hspace{.6ex}}c@{\hspace{.6ex}}|}{\raisebox{-.3ex}{$#1$}}}\raisebox{-.6ex}{$\begin{array}[b]{c}\cline{1-1}\lr{\overline{1}}\\\cline{1-1}\end{array}$}}$};
  \node (-2) at (34bp,86bp) [draw,draw=none] {${\def\lr#1{\multicolumn{1}{|@{\hspace{.6ex}}c@{\hspace{.6ex}}|}{\raisebox{-.3ex}{$#1$}}}\raisebox{-.6ex}{$\begin{array}[b]{c}\cline{1-1}\lr{\overline{2}}\\\cline{1-1}\end{array}$}}$};
  \draw [blue,->] (-2) ..controls (27.827bp,64.675bp) and (22.027bp,44.639bp)  .. (-1);
  \definecolor{strokecol}{rgb}{0.0,0.0,0.0};
  \pgfsetstrokecolor{strokecol}
  \draw (34bp,48bp) node {$1$};
  \draw [red,->] (2) ..controls (32.558bp,140.79bp) and (33.078bp,121.03bp)  .. (-2);
  \draw (42bp,124bp) node {$2$};
  \draw [blue,->] (1) ..controls (16.892bp,216.68bp) and (22.429bp,196.64bp)  .. (2);
  \draw (32bp,200bp) node {$1$};
  \draw [black,<-] (1) ..controls (10.211bp,193.74bp) and (9.6462bp,150.8bp)  .. (10bp,114bp) .. controls (10.333bp,79.329bp) and (11.292bp,38.164bp)  .. (-1);
  \draw (19bp,124bp) node {$0$};
\end{tikzpicture}
\hspace{0.5cm}
}
&
\scalebox{.8}{
\hspace{0.5cm}
\begin{tikzpicture}[>=latex,line join=bevel,]
\node (1) at (8bp,314bp) [draw,draw=none] {${\def\lr#1{\multicolumn{1}{|@{\hspace{.6ex}}c@{\hspace{.6ex}}|}{\raisebox{-.3ex}{$#1$}}}\raisebox{-.6ex}{$\begin{array}[b]{c}\cline{1-1}\lr{1}\\\cline{1-1}\end{array}$}}$};
  \node (0) at (27bp,162bp) [draw,draw=none] {${\def\lr#1{\multicolumn{1}{|@{\hspace{.6ex}}c@{\hspace{.6ex}}|}{\raisebox{-.3ex}{$#1$}}}\raisebox{-.6ex}{$\begin{array}[b]{c}\cline{1-1}\lr{0}\\\cline{1-1}\end{array}$}}$};
  \node (2) at (31bp,238bp) [draw,draw=none] {${\def\lr#1{\multicolumn{1}{|@{\hspace{.6ex}}c@{\hspace{.6ex}}|}{\raisebox{-.3ex}{$#1$}}}\raisebox{-.6ex}{$\begin{array}[b]{c}\cline{1-1}\lr{2}\\\cline{1-1}\end{array}$}}$};
  \node (-1) at (45bp,10bp) [draw,draw=none] {${\def\lr#1{\multicolumn{1}{|@{\hspace{.6ex}}c@{\hspace{.6ex}}|}{\raisebox{-.3ex}{$#1$}}}\raisebox{-.6ex}{$\begin{array}[b]{c}\cline{1-1}\lr{\overline{1}}\\\cline{1-1}\end{array}$}}$};
  \node (-2) at (23bp,86bp) [draw,draw=none] {${\def\lr#1{\multicolumn{1}{|@{\hspace{.6ex}}c@{\hspace{.6ex}}|}{\raisebox{-.3ex}{$#1$}}}\raisebox{-.6ex}{$\begin{array}[b]{c}\cline{1-1}\lr{\overline{2}}\\\cline{1-1}\end{array}$}}$};
  \draw [red,->] (2) ..controls (29.884bp,216.79bp) and (28.844bp,197.03bp)  .. (0);
  \definecolor{strokecol}{rgb}{0.0,0.0,0.0};
  \pgfsetstrokecolor{strokecol}
  \draw (38bp,200bp) node {$2$};
  \draw [black,<-] (1) ..controls (5.6096bp,269.71bp) and (4.071bp,226.75bp)  .. (6bp,190bp) .. controls (6.8901bp,173.04bp) and (7.4676bp,168.79bp)  .. (10bp,152bp) .. controls (13.03bp,131.91bp) and (17.947bp,108.64bp)  .. (-2);
  \draw (15bp,200bp) node {$0$};
  \draw [blue,->] (-2) ..controls (25.381bp,66.439bp) and (27.862bp,50.919bp)  .. (32bp,38bp) .. controls (32.982bp,34.934bp) and (34.219bp,31.766bp)  .. (-1);
  \draw (41bp,48bp) node {$1$};
  \draw [red,->] (0) ..controls (25.126bp,147.03bp) and (24.393bp,140.15bp)  .. (24bp,134bp) .. controls (23.409bp,124.75bp) and (23.143bp,114.47bp)  .. (-2);
  \draw (33bp,124bp) node {$2$};
  \draw [blue,->] (1) ..controls (14.454bp,292.68bp) and (20.517bp,272.64bp)  .. (2);
  \draw (31bp,276bp) node {$1$};
  \draw [black,<-] (2) ..controls (46.995bp,217.47bp) and (48.845bp,213.76bp)  .. (50bp,210bp) .. controls (71.525bp,139.95bp) and (54.171bp,49.504bp)  .. (-1);
  \draw (69bp,124bp) node {$0$};
\end{tikzpicture}
\hspace{0.5cm}
}
\\
$B^{1,1}$ for $C_2^{(1)}$  & $B^{1,1}$ for $B_2^{(1)} (\simeq B^{2,1}$ for $C_2^{(1)}$)
\end{tabular}
\end{center}
 
\caption{Examples of perfect and non-perfect $U_q'(\g)$ crystals. One can verify that the crystal on the right 
is perfect of level $1$, and the crystal on the left is not perfect of any level. These crystals are both KR crystals, 
as indicated, and the vertices are indexed by KN tableaux as in Section~\ref{ss:KN}. In fact, all known finite 
$U_q'(\g)$ crystals are KR crystals or tensor products of KR crystals. 
\label{figure:KR example}}
\end{figure}
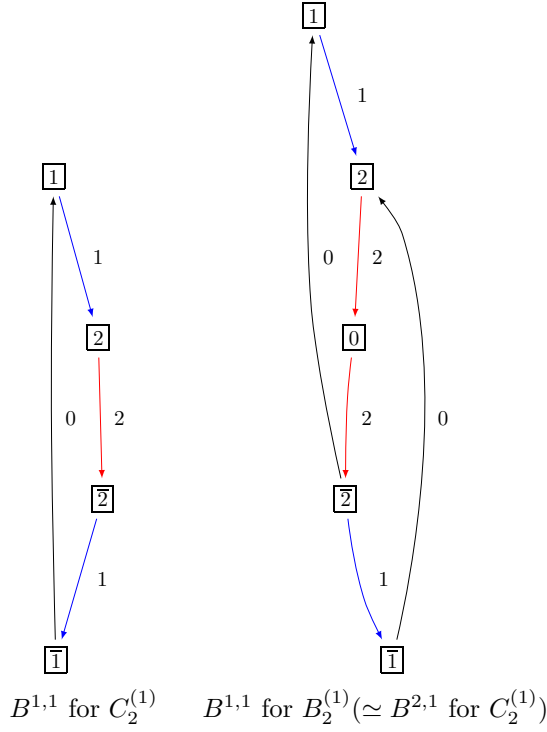
 
\subsection{Kirillov--Reshetikhin modules and their crystals}
\label{section.KR}

In this section $\g$ is of affine type.
The Kirillov--Reshetikhin modules were first introduced for the Yangian of $\g'$ in~\cite{KR:1987}.
One can characterize the KR module $W^{r,s}$ for $U_q'(\g)$~\cite{CP:1994,CP:1998},
where $r\in I\setminus \{0\}$ and
$s\ge 1$, as the irreducible representations of $U_q'(\g)$ whose Drinfeld polynomials are given by
\begin{equation}
P_i(u) =
\begin{cases}
(1-q_i^{1-s}u) (1-q_i^{3-s}u) \cdots (1-q_i^{s-1}u) & \text{if } i=r,\\
1 & \text{otherwise.}
\end{cases}
\end{equation}
Here $q_i = q^{(\alpha_i \mid \alpha_i)/2}$.
 
It was shown in~\cite{KMN2:1992} for type $A_n^{(1)}$ and many special cases, and
in~\cite{Okado:2007, OS:2008} for general non-exceptional types, that
the modules $W^{r,s}$ have crystal bases. We denote the resulting crystals by $B^{r,s}$, and
refer to them as {\it KR crystals}. 
 
\begin{theorem} \cite{OS:2008, FOS:2010}
In all non-exceptional types, $W^{r,s}$ has a crystal base $B^{r,s}.$ Furthermore,
if $s$ is a multiple of $c_r$ (see Figure \ref{u-table}) the resulting crystals are perfect. \qed
\end{theorem}
 
The following is a slight strengthening of results from~\cite{FOS:2009}.
 
\begin{lemma} \label{uniqueness}
Let $\g$ be a non-exceptional affine Kac--Moody algebra with index set $I=\{0,1,\ldots,n\}$.
Fix $r \in I \setminus \{0\}$ and $s>0$. Then any regular abstract crystal $B$ (see
Definition~\ref{abstract-C}) of type $\g$ which is isomorphic to $B^{r,s}$ as a
$\{1,2, \ldots, n \}$-crystal is also isomorphic to $B^{r,s}$ as an $I$-crystal.
\end{lemma}
 
\begin{proof}
First, recall that any finite type crystal $C$ is uniquely determined by its character
\begin{equation*}
        \ch(C):= \sum_{c \in C} e^{\wt^{\text{fin}}(c)},
\end{equation*}
where $\wt^{\text{fin}}$ denotes the weight as a $\{1,2,\ldots, n\}$-crystal.
Since $B$ is isomorphic to $B^{r,s}$ as a $\{1,2, \ldots ,n\}$-crystal, it is finite and hence of level $0.$
Thus $\wt(b)$ can be recovered from $\wt^{\text{fin}}(b)$. In particular, we can recover the character of
$B$ as an $I \backslash \{ j \}$ crystal for all $j \in I$, and from there recover the isomorphism class
of $B$ as an $I \backslash \{ j \}$-crystal.
 
The lemma now follows by the following uniqueness results from~\cite{OS:2008, FOS:2009}:
\begin{itemize}
\item
By~\cite[Proposition 6.1]{OS:2008} (for most nodes) and \cite{FOS:2009} (for exceptional nodes),
in types $D_n^{(1)}$, $B_n^{(1)}$ and  $A_{2n-1}^{(2)}$, any regular abstract crystal $B$ which
is isomorphic to $B^{r,s}$ as both a $\{ 1,2, \ldots, n \}$-crystal and a $\{ 0,2, \ldots, n \}$-crystal
is isomorphic to $B^{r,s}$ as an affine crystal.
\item
As in~\cite[Sections 5.2 and 6.1]{FOS:2009}, in types $C_n^{(1)}$, $D_{n+1}^{(2)}$ and $A_{2n}^{(2)}$,
any regular abstract crystal $B$ which is isomorphic to $B^{r,s}$ as both a $\{ 1,2, \ldots, n \}$-crystal and
a $\{ 0,1, \ldots, n-1 \}$-crystal is isomorphic to $B^{r,s}$ as an affine crystal.
\end{itemize}
\end{proof}
 
By~\cite[Proposition 3.8]{LOS:2010}, a tensor product $B=B^{r_1, s_1} \otimes \cdots \otimes
B^{r_N, s_N}$ of KR crystals is connected. We refer to such a $B$ as a {\it composite KR crystal}.
As in~\cite{KMN1:1992}, if the factors are all perfect KR crystals of the same level, then
$B=B^{r_1, \ell c_{r_1}} \otimes \cdots \otimes B^{r_N, \ell c_{r_N}}$ is also perfect of level $\ell$.
We refer to such a perfect crystal as a {\it composite KR crystal of level $\ell$}.
 
Explicit combinatorial models for KR crystals of non-exceptional type were constructed in~\cite{FOS:2009},
and will be reviewed in Section~\ref{section:realization}. Two examples are given in Figure~\ref{figure:KR example}.
 
\subsection{Extended affine Weyl group}
Write the null root as $\delta = \sum_{i \in I } a_i \alpha_i$. Let $\theta= \delta-a_0 \alpha_0$. As 
in~\cite[Section 6.4]{Kac:1990}, $\theta$ is a root in the finite type root system corresponding to 
$I \backslash \{ 0 \}$. Following~\cite{HKOTT:2002},
for each $i \in I \backslash \{ 0 \}$, define $c_i = \max(1,a_i/a_i^{\vee})$. It turns out that
$c_i=1$ except for $c_i=2$ for: $\g=B_n^{(1)}$ and $i=n$, $\g=C_n^{(1)}$ and $1\le i\le n-1$,
$\g=F_4^{(1)}$ and $i=3,4$, and $c_2=3$ for $\g=G_2^{(1)}$.
Here we use Kac's indexing of affine Dynkin diagrams from~\cite[Table Fin, Aff1 and Aff2]{Kac:1990}.
Consider the sublattices of $\oP$ given by
\begin{align*}
  M &= \bigoplus_{i\in  I \backslash \{ 0 \}} \Z c_i\alpha_i = \Z \oW \cdot \theta/a_0, \\
  \widetilde M &= \bigoplus_{i\in  I \backslash \{ 0 \}} \Z c_i\omega_i.
\end{align*}
Let $\oW$ be the finite type Weyl group for the Dynkin diagram $I \backslash \{ 0 \}$, which
acts on $\oP$ by linearizing the rules $s_i \lambda = \lambda - \langle \alpha_i^\vee, \lambda
\rangle \alpha_i$.
Clearly $M\subset \widetilde M$ and the action of $\oW$
on $\oP$ restricts to actions on $M$ and $\widetilde M$. Let $T(\widetilde M)$ (resp.
$T(M)$) be the subgroup of $T(\oP)$ generated by the translations $t_\lambda$ by
$\lambda \in\widetilde M$ (resp. $\lambda \in M$).
 
There is an isomorphism \cite[Prop. 6.5]{Kac:1990}
\begin{align} \label{E:Waffiso}
W \cong \oW \ltimes T(M)
\end{align}
as subgroups of $\Aut(\overline{P})$, where $W$ is the affine Weyl group.
Under this isomorphism we have
\begin{align} \label{E:s0}
s_0 = t_{\theta/a_0} s_\theta,
\end{align}
where $s_\theta$ is the reflection corresponding to the root $\theta$. Define $\theta^\vee \in \mathfrak{h}^*$ 
so that $s_\theta(\lambda)= \lambda -  \langle \theta^\vee, \lambda \rangle \theta$.

Define the extended affine Weyl group to be the subgroup of
$\Aut(P)$ given by
\begin{align} \label{E:WXdef}
\widetilde W = \oW \ltimes T(\widetilde M).
\end{align}
Let $C\subset \overline P \otimes_{\Z} \mathbb{R}$ be the fundamental chamber,
the set of elements $\lambda$ such that $\langle \alpha^\vee_i, \lambda \rangle \ge 0$
for all $i\in I \backslash \{ 0 \}$, and $\langle \theta^\vee, \lambda \rangle \leq 1/{a_0}.$ 

Let $\Sigma \subset \widetilde W$ be the subgroup of $\widetilde W$ consisting
of those elements that send $C$ into itself. Then $\widetilde W= W \Sigma$, and in particular
every element $x \in \widetilde W$ can be written uniquely as
\begin{equation}
	x= w \tau
\end{equation}
for some $w \in W$ and $\tau \in \Sigma$.
 
The usual affine Weyl group $W$ is
a normal subgroup of $\widetilde W$, so $\Sigma$ acts on $W$ by
conjugation. Each $\tau\in\Sigma$
induces an automorphism (also denoted $\tau$) of the affine Dynkin diagram $\Gamma$, which is characterized
as the unique automorphism so that:
\begin{align}\label{E:autoW}
\tau s_i \tau^{-1} = s_{\tau(i)} \qquad\text{for each $i\in I$.}
\end{align}
 
\begin{remark}
When $\g$ is of untwisted type, $M\cong Q^\vee$, $\widetilde M \cong P^\vee$,
with the isomorphism $\nu$ given by $c_i\omega_i = \nu(\omega_i^\vee)$, and
$c_i\alpha_i=\nu(\alpha_i^\vee)$ for $i\in I \backslash \{ 0 \}$.
\end{remark}
 
\begin{lemma} \label{Sigma-alt} Fix $\g$ with affine Dynkin diagram $\Gamma$, and let $\tau \in \Aut(\Gamma).$ Then
$\tau \in \Sigma$ if and only if there exists $w_\tau \in \oW$ such that the following diagram commutes
  \[
  \xymatrix{
    I \ar[r]^\iota \ar[d]^\tau & \Delta \ar[d]^{w_\tau} \\
      I \ar[r]_{\iota} & \Delta,
    }
  \]
where $\Delta$ is the underlying finite type root system, and $\iota$ is the map that takes $i$ to $\alpha_i$ for all $i \neq 0$
and takes $0$ to $- \theta$.
\end{lemma}
 
\begin{proof}
Fix $\tau \in \Sigma$. By~\eqref{E:WXdef}, we can write $\tau =  t_\lambda w$, where
$w \in \oW$, and $\lambda \in P^\vee$.  Then $w$ must send $C$ to a chamber which can be shifted back to $C$, 
from which one can see that $w$ has the desired properties.
 
Now fix $\tau \in \Aut(\Gamma)$. If there is a $w_\tau \in \oW$ making the diagram
commute, then consider the element $t_{\omega_{\tau(0)}} w_\tau\in \widetilde W$. A simple calculation 
shows that $t_{\omega_{\tau(0)}} w_\tau( C)=C$, and that this realizes $\tau$ as an element 
of $\Sigma$. This last step uses the fact that $a_{\tau(0)}$ is always 1 when there is a non-trivial diagram automorphism $\tau$. 
\end{proof}
 
\subsection{Demazure modules and crystals} \label{DC-sec}
In this section $\g$ is an arbitrary symmetrizable Kac-Moody algebra.
Let $\lambda$ be a dominant integral weight for $\g$. Define
\[
        W^\lambda:= \{ w \in W \mid w \lambda = \lambda \}.
\]
Fix $\mu\in W\lambda$, and recall that the $\mu$ weight space in $V(\lambda)$ is
one-dimensional. Let $u_\mu$ be a non-zero element of the $\mu$ weight space in $V(\lambda)$.
Write $\mu=w\lambda$ where $w$ is the shortest element in the coset $w W^\la$.
The {\em Demazure module} is defined to be
$$
V_w(\lambda) := U_q(\g)^{>0}\cdot u_{w(\lambda)},
$$
and the {\em Demazure character} is
\begin{equation*}
        \ch V_w(\lambda) = \sum_\mu \dim(V_w(\lambda)_\mu) e^\mu,
\end{equation*}
where $V_w(\lambda)_\mu$ is the $\mu$ weight space of the Demazure module
$V_w(\lambda)$.

It was conjectured by Littelmann~\cite{Littelmann:1995a} and proven by Kashiwara~\cite{Kashiwara:1993} that 
the intersection of a crystal basis of $V(\lambda)$ with $V_w(\lambda)$ is a crystal basis for $V_w(\lambda)$. 
The resulting subset $B_w(\lambda) \subset B(\lambda)$ is referred to as the {\em Demazure crystal}. 
It has the properties that it is closed under the action of the crystal operators $e_i$ (but not $f_i$), and that 
\begin{equation}
\label{equation:Kashiwara-Demazure}
       \ch V_w(\lambda) = \sum_{b\in B_w(\lambda)} e^{\wt(b)}.
\end{equation}

Define the set
\begin{align} \label{E:fw}
f_w(b) := \{\,f_{i_N}^{m_N}\dotsm f_{i_1}^{m_1}(b)\mid
m_k\in \Z_{\ge0}\},
\end{align}
where $w=s_{i_N}\dotsm s_{i_1}$ is any reduced decomposition
of $w$. By \cite[Proposition 3.2.3]{Kashiwara:1993}, as sets, 
\begin{align} \label{E:fwDem}
B_w(\lambda) = f_w(u_{\lambda}),
\end{align}
independent of the reduced word for $w$.

\begin{example}
Let us consider the Demazure crystal $B_{s_2 s_1}(\omega_1+\omega_2)$ for $\mathfrak{sl}_3$. 
The Demazure crystal is shown with thick vertices and edges, will the rest of the ambient crystal 
$B(\omega_1+\omega_2)$ is shown in thinner lines. Here blue lines show the action of $f_1$ 
and red lines show the action of $f_2$. 
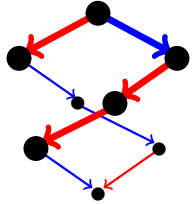
\begin{center}
\begin{tikzpicture}[scale=0.15]

\draw[blue, line width = 0.03cm, ->] (0,12) ..controls (2.5,10.25) and (2.5,10.25)  .. (5,8.5);
\draw[blue, line width = 0.03cm, ->] (5,8) ..controls (8.5,6.25) and (8.5,6.25)  .. (12,4.5);
\draw[blue, line width = 0.03cm, ->] (1.5,4) ..controls (4,2.25) and (4,2.25)  .. (6.5,0.5);
\draw[red, line width = 0.03cm, ->] (12.5,4) ..controls (10,2.25) and (10,2.25)  .. (7.5,0.5);

\draw[red, line width = 0.09cm, ->] (7,16) ..controls (3.75,14.25) and (3.75,14.25)  .. (0.5,12.5);
\draw[red, line width = 0.09cm, ->] (14,12) ..controls (11.5,10.25) and (11.5,10.25)  .. (9,8.5);
\draw[red, line width = 0.09cm, ->] (9,8) ..controls (5.5,6.25) and (5.5,6.25)  .. (2,4.5);\

\draw[blue, line width = 0.09cm, ->] (7,16) ..controls (10.25,14.25) and (10.25,14.25)  .. (13.5,12.5);

\draw (7.6,0) node {\circle*{5}};
\draw (1.5,4) node [circle, fill=black] {};
\draw (13,4) node {\circle*{5}};
\draw (5.8,8) node {\circle*{5}};
\draw (8.5,8) node [circle, fill=black] {};
\draw (0,12) node [circle, fill=black] {};
\draw (14,12) node [circle, fill=black] {};
\draw (7,16) node [circle, fill=black] {};

\end{tikzpicture}
\end{center}
\end{example}

\begin{remark}
We mainly consider the case when $\g$ is affine, and the following Demazure modules: Fix an anti-dominant 
weight $\mu \in \overline{P}$, and write $t_\mu \in \widetilde W$ as $t_\mu = w\tau$, where $w \in W$ and 
$\tau \in \Sigma$. For a dominant $\Lambda \in P$, we consider $V_w(\tau(\Lambda))$ and its crystal 
$B_w(\tau(\Lambda))$. In \cite{FSS:2007} these were denoted by $V_\mu(\Lambda)$ and 
$B_\mu(\Lambda)$, respectively. 
\end{remark}
 
 \begin{remark}
It is well-known~\cite{Demazure:1974,Kumar:1987,Mathieu:1986} that the Demazure
character can be expressed in terms of the Demazure operator $D_i:\Z[P] \to \Z[P]$
\begin{equation*}
        D_i(e^\mu) = \frac{e^{\mu} - e^{\mu-(1+\langle \alpha_i^\vee,\mu\rangle)\alpha_i}}{1-e^{-\alpha_i}},
\end{equation*}
where $\alpha_i$ is a simple root and $\alpha_i^\vee$ the corresponding coroot.
Then for $w=s_{i_N}\cdots s_{i_1}\in W$ a reduced expression
\begin{equation*}
        \ch V_w(\lambda) = D_{i_N} \cdots  D_{i_1} (e^\lambda).
\end{equation*}
\end{remark}
 
\section{Kashiwara-Nakashima tableaux and branching rules}
\label{section.finite.type}
 
\subsection{Kashiwara-Nakashima tableaux} \label{ss:KN}
 
We briefly review a method, due to  Kashiwara and Nakashima \cite{KN:1994}, of
realizing all highest weight crystals of non-exceptional finite type $\g$. The first observation is that many crystals
occur inside high enough tensor powers of the ``standard" crystals shown in
Figure~\ref{standard-crystals}. Many here means all in type $A$ and $C$, but not those
involving spin weights in types $B$ and $D.$
 
\begin{figure}
 
$$\begin{array}{|l|c|}
       \hline
               \text{Type}&\text{Standard crystal $B(\omega_1)$} \\
       \hline
               A_n&\text{\xymatrix@R=1ex{
                       \hbox{$\tab{1}$}\ar[r]^1& \hbox{$\tab{2}$}  \ar[r]^{2} & \cdots  \ar[r]^{n-1} & \hbox{$\tab{n}                          $}  \ar[r]^{n} &\hbox{$\tab{n+1}$}  }} \\
       \hline
               B_n&\text{\xymatrix@R=1ex{
                       \hbox{$\tab{1}$}\ar[r]^1& \cdots \ar[r]^{n-1}  & \hbox{$\tab{n}$} \ar[r]^{n} & \hbox{$\tab{0}$}                         \ar[r]^{n} & \hbox{$\tab{\overline{n}}$}         \ar[r]^{n-1} & \cdots \ar[r]^1&  \hbox{$\tab                                   {\overline{1}}$} } }
                       \\
       \hline
               C_n&\text{\xymatrix@R=1ex{
                       \hbox{$\tab{1}$}\ar[r]^1& \cdots \ar[r]^{n-1}  & \hbox{$\tab{n}$} \ar[r]^{n} &
                       \hbox{$\tab{\overline{n}}$}      \ar[r]^{n-1} & \cdots \ar[r]^1&  \hbox{$\tab                                                   {\overline{1}}$} } }  \\
       \hline
               D_n&\text{\xymatrix@R=1ex{
                       &&&\hbox{$\tab{n}$}\ar[dr]^n\\
                       \hbox{$\tab{1}$}\ar[r]^1&\cdots\ar[r]^{n-2}&\hbox{$\tab{n-1}$}\ar[ur]^{n-1}\ar[dr]_n&&
                       \hbox{$\tab{\overline{n-1}}$}\ar[r]^{n-1}&\cdots\ar[r]^1&\hbox{$\tab{\overline{1}}$}\\
                       &&&\hbox{$\tab{\overline{n}}$}\ar[ur]_{n-1}}} \\
       \hline
\end{array}$$
 
\caption{Standard crystals. The boxes represent the vertices of the crystal, and each arrow labeled $i$ shows the
action of $f_i$.  \label{standard-crystals}}
\end{figure}
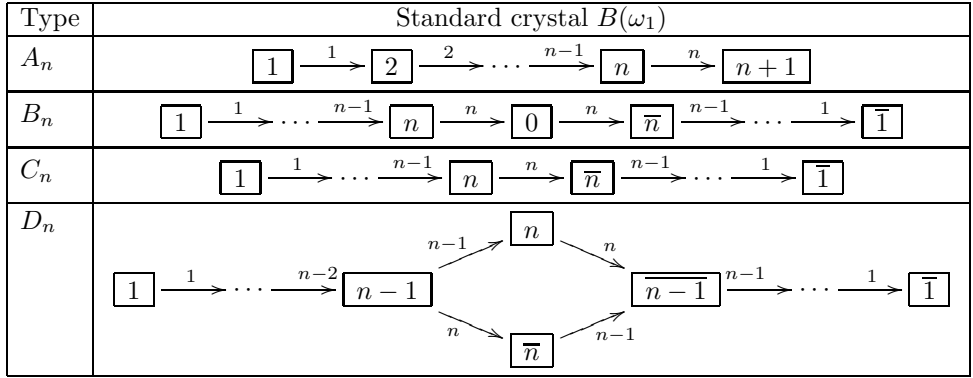
 
We call the set of symbols that show up in the boxes of the standard crystal of type
$X_n= A_n, B_n, C_n, D_n$ the type $X_n$ alphabet. Impose a partial order $\prec$ on
this alphabet by saying $x \prec y $ iff $x$ is to the left of $y$ in the presentation of the
standard crystals in Figure~\ref{standard-crystals} (in type $D_n$, the symbols $n$ and
$\overline{n}$ are incomparable).
 
\begin{definition} \label{Lg}
Fix $\g$ of type $X_n$, for $X=A,B,C,D$. Fix a dominant integral weight $\gamma$ for $\g=X_n$.
Write $\gamma = m_1 \omega_1 + m_2 \omega_2 + \cdots + m_{n-1}  \omega_{n-1} + m_n \omega_n$.
Define a generalized partition $\Lambda(\gamma)$ associated to $\gamma$, which is defined case
by case as follows:
\begin{itemize}
\item If $X=A,C$, $\Lambda(\gamma)$ has  $m_i$ columns of height $i$ for each $ 1 \leq i \leq n$;
\item If $X=B$, $\Lambda(\gamma)$ has $m_i$ columns of height  $i$ for each $ 1 \leq i \leq n-1$,
        and $m_n/2$ columns of height $n$;
\item If $X=D$, $\Lambda(\gamma)$ has  $m_i$ columns of each height $i$ for each $ 1 \leq i \leq n-2$,
        $\min(m_{n-1},m_n)$ columns of height $n-1$, and $|m_n-m_{n-1}|/2$ columns of height $n$.
        Color columns of height $n$ using color 1 if $m_n> m_{n-1}$ and color 2 if      $m_n<m_{n-1}$.
\end{itemize}
We use French notation for partitions here, where we adjust the columns at the bottom.
In cases where the above formulas involve a fractional number $x$ of columns at some height, we denote
this by putting $\lfloor x \rfloor$ columns in addition to a single column of half width. Notice that this can only
happen for columns of height $n$, and at worst we get a single column of width $1/2$.
\end{definition}
 
\begin{definition} \label{non-spin:def}
Fix a dominant integral weight $\gamma$. We say $\gamma$ is a \textit{non-spin weight} if
\begin{itemize}
\item  In types $A_n$ and $C_n$, no conditions.
 
\item In type $B_n$,  $\Lambda(\gamma)$ does not contain any column of width $1/2$ (or equivalently, $m_n$ is even).
 
\item In type $D_n$, $\Lambda(\gamma)$ has no columns of height $n$ (or equivalently, $m_{n-1}=m_n$).
\end{itemize}
\end{definition}
 
In the case when $\Lambda=\Lambda(\gamma)$ does not contain a column of width
$1/2$, the highest weight crystal $B(\gamma)$ embeds in the $M$-th tensor power of
$B(\omega_1)$, where $M$ is the number of boxes in the partition $\Lambda$. Furthermore, the
image of $B(\gamma)$ is contained in the set of $ x_M  \otimes \cdots \otimes x_1$ such, when
entered into the Young diagram $\Lambda$ in order moving up columns and right to left
(see Figure~\ref{make-a-tableau}), the result is weakly increasing along rows and, if you ignore
the symbols $0$ in type $B$ and $n, \bar n$ in type $D$, also strictly increasing up columns. We refer to fillings 
which occur in the image of $B(\gamma)$ as Kashiwara--Nakashima (KN) tableaux. Precise conditions describing 
KN tableaux are given in \cite{KN:1994}. The crystal structure on the set of KN tableaux is inherited from the crystal 
structure on $B(\omega_1)^{\otimes M}.$ Modifications of this construction dealing with spin weights are given 
in~\cite{KN:1994}.
 
\begin{remark}
The above construction of $B(\lambda)$ in type $D_n$ goes through without modification in the case $n=3$,
when $D_3 \cong A_3$. Thus we have two realizations of the crystal in this case.
\end{remark}
 
\begin{figure}
\setlength{\unitlength}{0.5cm}
\begin{center}
 
\begin{picture}(22,3)
\put(1,1){\hbox{$\tab{\overline{4}}$}}
\put(2.25, 1.0){$\otimes$}
\put(3,1){$\tab{3}$}
\put(4.25, 1.0){$\otimes$}
\put(5,1){$\tab{1}$}
\put(6.25, 1.0){$\otimes$}
\put(7,1){$\tab{\overline{3}}$}
\put(8.25, 1.0){$\otimes$}
\put(9,1){$\tab{2}$}
\put(10.25, 1.0){$\otimes$}
\put(11,1){$\tab{\overline{3}}$}
\put(12.25, 1.0){$\otimes$}
\put(13,1){$\tab{3}$}
\put(15.5,1){$\longrightarrow$}
 
\put(18,0){
\begin{picture}(3, 3)
\put(0,0){\line(0,1){3}}
\put(0.05,0){\line(0,1){3}}
 
\put(1,0){\line(0,1){3}}
\put(1.05,0){\line(0,1){3}}
 
\put(2,0){\line(0,1){2}}
\put(2.05,0){\line(0,1){2}}
 
\put(3,0){\line(0,1){2.05}}
\put(3.05,0){\line(0,1){2.05}}
 
\put(0,0){\line(1,0){3}}
\put(0,0.05){\line(1,0){3}}
 
\put(0,1){\line(1,0){3}}
\put(0,1.05){\line(1,0){3}}
 
\put(0,2){\line(1,0){3.05}}
\put(0,2.05){\line(1,0){3.05}}
 
\put(0,3){\line(1,0){1.05}}
\put(0,3.05){\line(1,0){1.05}}
 
\put(0.35,0.3){$1$}
\put(1.35,0.3){$2$}
\put(2.35,0.3){$3$}
 
\put(0.35,1.3){$3$}
\put(1.35,1.3){$\overline{3}$}
\put(2.35,1.3){$\overline{3}$}
 
\put(0.35,2.3){$\overline{4}$}
 
\end{picture}
}
 
\end{picture}
\end{center}
 
\caption{An element of $B(\omega_3 + 2 \omega_2)$ of type $C_4$
as realized inside the tensor product $B(\omega_1)^{\otimes 7}$ of seven copies of the standard
crystal, and the corresponding tableau.
For explicit conditions on which tableaux appear in this correspondence, see~\cite{KN:1994}.
\label{make-a-tableau}}
\end{figure}
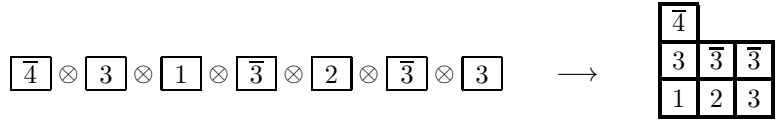
 
\subsection{Branching rules and $\pm$ diagrams}
\label{pm-diagrams}
We now describe branching rules for certain representations of $X_n$, where $X=B,C$ or $D$.
 
\begin{definition} \label{pm-def}
Fix a non-spin weight $\gamma$ (see Definition \ref{non-spin:def}), and set $\Lambda= \Lambda(\gamma)$. A
{\it $\pm$ diagram} $P$ with outer shape $\Lambda$ is a sequence of partitions
$\la\subseteq  \mu \subseteq \La$ such that $\La/\mu$ and $\mu/\la$ are horizontal strips
(i.e. every column contains at most one box). We depict a $\pm$ diagram by filling $\mu/\la$
with the symbol~$+$ and those of $\La/\mu$ with the symbol~$-$. We make the additional type dependent
modifications and restrictions:
\begin{enumerate}
\item
In type $C_n$, $\lambda$ has no columns of height $n$.
\item
In type $B_n$, $\lambda$ has no columns of height $n$. Additionally, we allow the
$\pm$ diagram to contain an extra symbol $0$. There can be at most one $0$, this must occur at
height $n$, must be to the right of all $+$ at height $n$, and to the left of all $0$ at height $n$,
and must be the only symbol in its column.
\item \label{pm-D}  In type $D_n$, columns of $\lambda$ of height $n-1$ are either all colored 1 or all colored 2.
\end{enumerate}
We denote by $\os(P)=\Lambda$ the outer shape of $P$ and by $\is(P)=\lambda$ the inner
shape of $P$.
\end{definition}
 
Embed $X_{n-1}$ into $X_n$ by removing node $1$ from the Dynkin diagram of type $X_n$. In the
special cases of $B_2, C_2$ and $D_3$, removing the node $1$ gives a Dynkin
diagram of a different type ($A_1, A_1,$ and $A_1 \times A_1$, respectively). By
abuse of notation, we use the symbol $X_{n-1}$ to mean this new diagram in these special cases.
Although these special cases are not explicitly mentioned in~\cite{FOS:2009}, the proof of
Theorem~\ref{branching} goes through without change.
 
\begin{theorem} (see \cite[Section 3.2]{FOS:2009}) \label{branching}
Fix a non-spin dominant integral weight $\gamma$ (see Definition~\ref{non-spin:def}) for $\g=X_n$, and let
$\Lambda = \Lambda(\gamma)$ (see Definition~\ref{Lg}). Then there is a
bijection between $\pm$ diagrams $P$ with outer shape $\Lambda$ and $X_{n-1}$
highest weight elements in $B(\Lambda)$. This can be realized in terms of KN tableaux
by the following algorithm.
\begin{enumerate}
        \item For each $+$ at height $n$,  fill that column with $12\ldots n$.
        \item Replace each $-$ with a $\ol{1}$ and, if there is a $0$ in the $\pm$ diagram, place a
        $0$ in that position of the tableaux.
        \item Fill the remainder of all columns by strings of the form
        $23\ldots k$.
        \item Let $S$ be the multi-set containing the heights of all the $+$ in $P$ of height less than
        $n$. Move through the columns of $\Lambda$  from top to bottom, left to right, ignoring the
        columns $12\ldots n$ of step (i), modifying the tableaux as follows.
        Each time you encounter a $\ol{1}$, replace it with  $\ol{h+1}$, where $h$ is the largest
        element of $S$, and delete $h$ from $S$. Each time you encounter a $2$ which is at the bottom
        of a column, replace the string
        $23\ldots k$ in that column by $12\ldots h\; h+2\ldots k$, where $h$ is the largest element
        of $S$, and remove $h$ from $S$. Once $S$ is empty, stop.
        \item In type $D_n$, if the $\pm$ diagram has empty columns of height $n-1$ colored 2,
        change all occurrences of $n$ at height $n-1$ to $\overline{n}$.
\end{enumerate}
Furthermore, two $\pm$ diagrams $P$ and $P'$ correspond to $X_{n-1}$ highest weight vectors of the same
$X_{n-1}$ weight if and only if $\is(P)=\is(P')$.
\qed
\end{theorem}
 
\begin{remark}
Definition \ref{pm-def} part \eqref{pm-D} differs from the statement in~\cite{FOS:2009}, since in that paper there were
no colorings for columns of height $n-1$. These colorings are needed, as can be seen by considering the branching
rules for $B(\omega_{n-1}+ \omega_n)$ for $D_4$:
\begin{equation*}
\tableau[sbY]{3|2|1} \raisebox{0.5cm}{$\;\leftrightarrow\;$} \tableau[sbY]{+| | |} \qquad \quad
\tableau[sbY]{\overline{1}|3|2} \raisebox{0.5cm}{$\;\leftrightarrow\;$} \tableau[sbY]{-| | |} \qquad \quad
\tableau[sbY]{\overline{3}|3|2} \raisebox{0.5cm}{$\;\leftrightarrow\;$} \tableau[sbY]{-| + | |} \qquad \quad
\tableau[sbY]{4|3|2} \raisebox{0.5cm}{$\;\leftrightarrow\;$} \tableau[sbY]{| | |} \qquad \quad
\tableau[sbY]{\overline{4}|3|2} \raisebox{0.5cm}{$\;\leftrightarrow\;$} \tableau[sbY]{| | |}
\end{equation*}
where the fourth $\pm$ diagram is considered of color 1 and the fifth of color 2.
\end{remark}
 
\begin{definition} \label{definition.KRw}
Let $\gamma$ be an integral highest weight of type $X_n$ and write
$\gamma = m_1 \omega_1 + m_2 \omega_2 + \cdots + m_{n-1}  \omega_{n-1} + m_n \omega_n$.
We call $\gamma$ a {\it case (AUT) weight} if it satisfies:
\begin{itemize}
\item If $X=C$, no conditions;
\item If $X=B$, assume $m_n=0$;
\item If $X=D$, $m_n=m_{n-1}=0$.
\end{itemize}
Note that in all these cases, $\Lambda(\gamma)$ is an ordinary partition.
\end{definition}
 
\begin{remark}
The term ``case (AUT)" in Definition~\ref{definition.KRw} is motivated by the fact that these are exactly the
classical highest weights that appear in case (AUT) KR crystals, as defined in Section~\ref{section:realization}
below.
\end{remark}
 
\begin{remark} All case (AUT) weights are non-spin, as in Definition \ref{non-spin:def}. Furthermore, if one starts 
with a case (AUT) weight $\gamma$, then the $X_{n-1}$ weights which appear in the decomposition from 
Theorem~\ref{branching} will all be non-spin weights, and the $X_{n-1}$ highest weight $\gamma'$ corresponding 
to the $\pm$ diagram $P$ will satisfy $\Lambda(\gamma')=\is(P).$
\end{remark}
 
\subsection{Nested $\pm$ diagrams}
 
If $\gamma$ is a case (AUT) weight, $n \geq 3$ in type $B_n, C_n$, and $n \geq 4$ in type $D_n$,
it follows immediately from Section~\ref{pm-diagrams} that the branching rule from $X_n$ to
$X_{n-2}$ can be described using pairs of $\pm$ diagrams $P$ and $p$ with $\os(P) = \Lambda(\gamma)$ and
$\is(P) = \os(p)$. We call such pairs of $\pm$ diagrams \textit{nested}.
The following result from~\cite{S:2008} gives an explicit description of  the action of $e_1$ on an
$X_{n-2}$ highest weight vector in terms of pairs of $\pm$ diagrams.
Since $e_1$ commutes with all $e_j$ for $j \geq 3$, this completely describes $e_1$.
 
Pair off the symbols in $(P,p)$ according to the rules
\begin{itemize}
\item Successively run through all $+$ in $p$ from left to right and, if possible, pair it with the
leftmost yet unpaired $+$ in $P$ weakly to the left of it.
\item Successively run through all $-$ in $p$ from left to right and, if possible, pair it with the
rightmost yet unpaired $-$ in $P$ weakly to the left.
\item Successively run through all yet unpaired $+$ in $p$ from left to right and, if possible,
pair it with the leftmost yet unpaired $-$ in $p$.
\end{itemize}
 
Notice that a $\pm$ diagram is uniquely determined by its outer (or inner) shape along with the number of each
symbol $+$, $-$, and $0$ on each row. Similarly, a pair $(P,p)$ of nested $\pm$ diagrams is uniquely determined
by its intermediate shape $\ms(P,p)= \os(p) = \is(P)$, along with the data of how many of each symbol is in each
row of $P$ and of $p$.
Since $e_1$ commutes with $e_j$ for all $j \geq 3$, the following gives a complete description of the action of
$e_1$ on the crystal:
\begin{lemma} \cite[Lemma 5.1]{S:2008} \label{S5.1}
Let $\gamma$ be a case (AUT) weight and fix a pair $(P,p)$ of nested $\pm$ diagrams such that
$\os(P) = \Lambda(\gamma)$. Let $b$ be the corresponding $X_{n-2}$ highest weight vector
in $B(\gamma)$. If there are no unpaired $+$ in $p$ and no unpaired $-$ in $P$, then $e_1(b)=0$.
Otherwise, $e_1(b)$ is the $X_{n-2}$ highest weight element corresponding to the pair of $\pm$ diagrams
$(P',p')$ described as follows.
 
If there is an unpaired $+$ in $p$, let $k$ be the height the rightmost unpaired $+$ in $p$.
\begin{itemize}
\item If there is a $-$ directly above the chosen $+$ in $p$, then $\ms(P,p)/\ms(P',p')$ is a single box at height $k+1$.
All rows of $P$ and $p$ have the same number of each symbol except: There is one more + in $P$ at height $k+1$;
there is one less $-$ in $p$ at height $k+1$;  there is one less $+$ and one more $-$ in $p$ at height $k$.
 
\item Otherwise, $\ms(P,p)/\ms(P',p')$ is a single box at height $k$.  All rows of $P$ and $p$ have the same number
of each symbol except: There is one more $+$ in $P$ at height $k$; there is one less $+$ in $p$ at height $k$.
\end{itemize}
 
\noindent  Otherwise let $k$ be the height of the leftmost unpaired $-$ in $P$.
\begin{itemize}
\item If there is a $+$ directly below the chosen $-$ in $P$, then $\ms(P',p')/\ms(P,p)$ is a single box at height $k-1$.
All rows of $P$ and $p$ have the same number of each symbol except: There is one less $-$ and one more $+$ in 
$P$ at height $k$; there is one less $+$ in $P$ at height $k-1$;  there is one more $-$ in $p$ at height $k-1$.
 
\item Otherwise, $\ms(P',p')/\ms(P,p)$ is a single box at height $k$. All rows of $P$ and $p$ have the same number
of each symbol except: There is one less $-$ in $P$ at height $k$; there is one more $-$ in $p$ at height $k$.
 
\end{itemize}
In type $B$, in every case, the number of $0$s in $p$ remains unchanged ($P$ cannot contain $0$ since
we are assuming that $\gamma$ is in case (AUT)).
\end{lemma}
 
\begin{remark}
In principle, one could build a new combinatorial model for type $B_n$, $C_n$, $D_n$ crystals using nested
$\pm$ diagrams, where the crystal operators would be given by Lemma~\ref{S5.1}. That is, the tableaux would
consist of $n-1$ nested $\pm$ diagrams in type $B_n$, $C_n$ and $n-2$ nested $\pm$ diagrams in type
$D_n$, along with some extra data recording an element in a representation of $A_1$ (or $A_1 \times A_1$ in type
$D_n$). This would have some advantages over KN tableaux, in that branching rules would be more readily visible.
Such tableaux may also be more convenient for working with KR crystals.
\end{remark}

\section{Realizations of KR crystals}
\label{section:realization}
 
Explicit combinatorial models for KR crystals $B^{r,s}$ for the non-exceptional types
were constructed in~\cite{FOS:2009}. In this section we recall and prove some properties
that we need for the definition of the energy function in Section~\ref{section.energy} and
Lemma~\ref{prime-contribution}.
 
The results are presented for three different cases:
\begin{enumerate}
 
\item[(IRR)] \label{case-A}
The classically irreducible KR crystals $B^{r,s}$:\newline
\begin{tabular}{ll}
$A_n^{(1)}$ &  $1\le r\le n$\\
$D_n^{(1)}$ &  $r=n-1,n$\\
$D_{n+1}^{(2)}$ & $r=n$\\
$C_n^{(1)}$ &   $r=n$
\end{tabular}
 
\item[(AUT)] \label{case-B}
KR crystals $B^{r,s}$ constructed via Dynkin automorphisms:\newline
\begin{tabular}{ll}
$D_n^{(1)}$ & $1\le r \le n-2$\\
$B_n^{(1)}$ & $1\le r\le n-1$\\
$A_{2n-1}^{(2)}$ & $1\le r\le n$
\end{tabular}
 
\item[(VIR)] \label{case-C}
The remaining virtually constructed KR crystals:\newline
\begin{tabular}{ll}
$B_n^{(1)}$ &   $B^{n,s}$\\
$C_n^{(1)}$  &   $B^{r,s}$ for $1\le r<n$\\
$D_{n+1}^{(2)}$ &  $B^{r,s}$ for $1\le r<n$\\
$A_{2n}^{(2)}$ &  $B^{r,s}$ for $1\le r\le n$.
\end{tabular}
\end{enumerate}
In Section~\ref{class-decomp} we present the classical decomposition of the KR crystals in the various cases.
Case (IRR) is  discussed in Section~\ref{caseIRR:section}. This case is well-understood, and we simply state
the facts we need. Case (AUT) is handled in Section~\ref{caseAUT:section}. This time we discuss the realizations
in more detail, and prove a technical lemma. Case (VIR) is handled in Sections~\ref{vir1}, \ref{vir2} and~\ref{vir3}.
Following~\cite{FOS:2009}, we realize these KR crystals as virtual crystals, using a combination of the similarity
method of Kashiwara \cite{Kashiwara:1996} and the technique of virtual crystals developed
in~\cite{OSS:2003,OSS:2002,FOS:2009}. We refer to crystals obtained using both of these methods as virtual
crystals. 
 
\subsection{Classical decompositions of KR crystals} \label{class-decomp}
Set
\begin{equation} \label{diamond-eq}
\begin{split}
        \diamond = \begin{cases}
        \emptyset & \text{for type $A_n^{(1)}$ and $1\le r\le n$}\\
        & \text{for types $C_n^{(1)}, D_{n+1}^{(2)}$ and $r=n$}\\
        & \text{for type $D_n^{(1)}$ and $r=n-1,n$}\\
        \text{vertical domino}         & \text{for type $D_n^{(1)}$ and $1\le r\le n-2$}\\
        & \text{for types $B_n^{(1)}$, $A_{2n-1}^{(2)}$ and $1\le r \le n$}\\ 
        \text{horizontal domino}    & \text{for types $C_n^{(1)}$, $D_{n+1}^{(2)}$ and $1\le r<n$}\\
        \text{box}                              & \text{for type $A_{2n}^{(2)}$ and $1\le r \le n$.}
        \end{cases}
\end{split}
\end{equation}
Then every $B^{r,s}$ decomposes as a classical crystal as
\begin{equation} \label{KR-decomp}
        B^{r,s} \cong \bigoplus_{\lambda} B(\lambda),
\end{equation}
where the sum is over those $\lambda$ such that $\Lambda(\lambda)$ (see Definition \ref{Lg}) can be obtained 
from the rectangle  $\Lambda(s\omega_r)$ by removing some number of $\diamond$, each occurring with multiplicity 1. 
In the untwisted case this was obtained by Chari~\cite{Chari:2001}. In the
twisted case, it was conjectured in~\cite[Appendix A]{HKOTT:2002} and proven by Hernandez
(see~\cite[Section 5]{Hernandez:2010}).
 
\subsection{Classically irreducible KR crystals (case (IRR))} \label{caseIRR:section}
 
As a $\{1,2,\ldots,n\}$-crystal, $B^{r,s} \cong B(r\omega_s).$
For example, the realization of the KR crystal $B^{r,s}$ of type $A_n^{(1)}$ is well-known in
terms of rectangular Young tableaux of shape $(s^r)$. We refer the reader to
e.g.~\cite[Section 4.1]{FOS:2009} for details.
The construction of the other irreducible KR crystals can be found in~\cite[Section 6]{FOS:2009}.
We only need the following fact, which follows immediately from the explicit models.
 
\begin{lemma} \label{lemma:irreducible}
Let $B^{r,s}$ be a KR crystal of type $\g$ for one of the cases in case (IRR).
Then for all $b \in B^{r,s}$, $\varepsilon_0(b) \leq s$.
\end{lemma}

\subsection{KR crystals constructed via Dynkin automorphisms (case (AUT))}  
\label{caseAUT:section}
 
Let $\g$ be of type $D_n^{(1)}$, $B_n^{(1)}$ or $A_{2n-1}^{(2)}$, with the underlying finite type
Lie algebra of type $X_n=D_n$, $B_n$ or $C_n$, respectively. 
Fix $s>0$ and $r$ so that $B^{r,s}$ is in case (AUT). Consider the classical crystal
\[
C^{r,s}:= \bigoplus B(\lambda),
\]
where the sum is over all $\lambda$ such that $\Lambda(\lambda)$ (see
Definition~\ref{Lg}) can be obtained from $\Lambda(s \omega_r)$ by removing vertical
dominos. As in Section~\ref{pm-diagrams}, the $X_{n-1}$ highest weight elements in $C^{r,s}$ (i.e. the highest weight 
element for the algebra with Dynkin diagram $I \backslash \{ 0,1 \}$)
are indexed by $\pm$ diagrams whose outer shape can be obtained from $\Lambda(s \omega_r)$
by removing vertical dominos.  
 
\begin{definition} \label{varsigma-def}
Define the involution $\varsigma$ on the $X_{n-1}$ highest weight vectors of $C^{r,s}$, as
indexed by $\pm$ diagrams, as follows.
Let $P$ be a $\pm$ diagram with $\os(P)=\Lambda$ and $\is(P)=\lambda$. For each
$1 \leq i \leq r-1$:
\begin{enumerate}
\item If  $i\equiv r-1 \pmod{2}$ then above each column of $\la$ of height $i$, there must
be a $+$ or a $-$. Interchange the number of such $+$ and $-$.
\item If $i\equiv r \pmod{2}$ then above each column of $\la$ of height $i$,
either there are no signs or a $\mp$ pair. Interchange the number of columns of each type.
\end{enumerate}
By Theorem~\ref{branching}, this can be extended in a unique way to an involution
$\varsigma$ on $C^{r,s}$.
\end{definition}
 
\begin{theorem} \label{vertical-realization} \cite{S:2008,FOS:2009}
Fix $\g$ and a KR crystal $B^{r,s}$ in case (AUT).
Define operators $e_0$ and $f_0$ on $C^{r,s}$ by $e_0 := \varsigma \circ e_1 \circ \varsigma$ and
$f_0 := \varsigma \circ f_1 \circ \varsigma$. Then $C^{r,s}$ along with these new operators is
isomorphic to $B^{r,s}.$  \qed
\end{theorem}
 
\begin{lemma} \label{goode-D}
Fix $\g$ and a KR crystal $B^{r,s}$ as in case (AUT). Consider the realization of $B^{r,s}$ given in
Theorem~\ref{vertical-realization}.
Fix $b \in B^{r,s}$, and assume $b$ lies in the component $B(\gamma)$ and $e_0(b)$
lies in the classical component $B(\gamma')$. Then
\begin{enumerate}
\item $\Lambda(\gamma')$ is either equal to $\Lambda(\gamma)$, or else is obtained from
$\Lambda(\gamma)$ by adding or removing a single vertical domino.
\item
If $\varepsilon_0(b)>s$, then $\Lambda(\gamma')$ is obtained from $\Lambda(\gamma)$ by
removing a vertical domino.
\end{enumerate}
\end{lemma}
 
\begin{proof}
Recall that for $b\in B^{r,s}$ by definition $e_0(b) = \varsigma \circ e_1\circ \varsigma (b)$.
Since $e_0$ and $e_1$ commute with $e_i$ for $i=3,4,\ldots,n$, they are defined
on $X_{n-2}$ components, which are described by pairs of $\pm$ diagrams $(P,p)$.
Consider the $1$-string $\{b_0,b_1,\ldots,b_k\}$ where $b_i=e_1^i\circ \varsigma(b)$
with corresponding $\pm$ diagrams $(P_i,p_i)$.
By Lemma~\ref{S5.1} there exists a $0\le j\le k$ such that:
\begin{enumerate}
\item[(a)] For $0\le i<j$, $P_{i+1}$ is obtained from $P_i$ by the addition of one box containing $+$;
the $+$ are added from right to left with increasing $i$.
\item[(b)] For $j\le i<k$, $P_{i+1}$ is obtained from $P_i$ by the removal of one box containing $-$;
the $-$ are removed from left to right with increasing $i$.
\end{enumerate}
Recall that $\varsigma$ interchanges $+$ and $-$ in columns of height congruent to $1\pmod{r}$
and $\mp$-pairs and empty columns of height congruent to $0\pmod{r}$. Empty columns of
height $r$ are left unchanged.
Since the width of the diagrams is at most $s$, we can only have
$\varepsilon_0(b)>s$ if we are in case (a) above and the $+$ is added at height strictly smaller
than $r$. Then there are two cases:
\begin{itemize}
\item The $+$ is added below a $-$ in $P_0$.
\item The $+$ is added in an empty column in $P_0$.
\end{itemize}
In both cases it is easy to check from the rules of $\varsigma$ that a vertical
domino is removed from the outer shape.
 
In all other cases, applying $e_1$ to $\varsigma(b)$ only adds or removes a single symbol to/from the corresponding 
$\pm$ diagram $P$, so from the definition of $\varsigma$ and the fact that $e_0= \varsigma \circ e_1 \circ \varsigma$, 
it is clear that $\Lambda(\gamma')$ differs from $\Lambda(\gamma)$ by at most one vertical domino, and the lemma 
holds.
\end{proof}
 
\subsection{$B^{r,s}$ of type $C_n^{(1)}$ for $r<n$ (in case (VIR))}
\label{vir1}
 
The KR crystals of type $C_n^{(1)}$ are constructed inside an ambient crystal of type $A_{2n+1}^{(2)}$.
 
\begin{theorem} \cite[Section 4.3, Theorem 5.7]{FOS:2009} \label{virtual-KR-C}
Fix $n \geq 2$ and $r<n$. Let $\hat{B}^{r,s}$ be the KR crystal corresponding to type
$A_{2n+1}^{(2)}$, with crystal operators $\hat e_0, \ldots, \hat e_{n+1}$.
Let $V$ be the subset of all $b\in \hat{B}^{r,s}$ which are invariant under
the involution $\varsigma$ from Definition~\ref{varsigma-def}.
Define the virtual crystal operators
$e_i:= \hat e_{i+1}$ for $1 \leq i \leq n$ and $e_0:= \hat e_0 \hat e_1$.
Then $V$ along with the operators $e_0, \ldots, e_{n}$ is isomorphic to $B^{r,s}$ for type
$C_n^{(1)}$.
 
Here the operators $f_i$ and $\hat f_i$ are defined by the condition $f_i(b) = b'$ if and only
if $e_i(b')=b$. \qed
\end{theorem}

Recall that as a classical
crystal, $B^{r,s}$ of type $C_n^{(1)}$ decomposes as in~\eqref{KR-decomp}.
 
\begin{lemma}
\label{goode-C}
Let $B^{r,s}$ with $r<n$ be the KR crystal of type $C_n^{(1)}$.
Fix $b \in B^{r,s}$, and assume $b$ lies in the component $B(\gamma)$ and $e_0(b)$
lies in the classical component $B(\gamma')$. Then
\begin{enumerate}
\item \label{only-dominoes-gC} $\Lambda(\gamma')$ is either equal to $\Lambda(\gamma)$, or else is obtained from
$\Lambda(\gamma)$ by adding or removing a single horizontal domino.
\item 
If $\varepsilon_0(b)>\lceil s/2 \rceil$,
then $\Lambda(\gamma')$ is obtained from $\Lambda(\gamma)$ by
removing a horizontal domino.
\end{enumerate}
\end{lemma}
 
\begin{proof}
Realize the crystal $B^{r,s}$ of type
$C_n^{(1)}$ inside the ambient crystal $\hat{B}^{r,s}$ of type $A_{2n+1}^{(2)}$ using Theorem~\ref{virtual-KR-C}.
Denote the embedding by $S: B^{r,s} \hookrightarrow \hat{B}^{r,s}$, and set
$\hat{b}=S(b)$. Let $(P,p)$ be the pair of $\pm$ diagrams associated to
the $\{3,4,\ldots,n+1\}$ highest weight corresponding to $\hat{b}$, and let $\gamma$ be the highest weight of
$b$ as a $C_n \subset C_n^{(1)}$ crystal.
Since $e_0=\hat{e}_1 \hat{e}_0$ and $e_i=\hat{e}_{i+1}$, $\Lambda(\gamma)=\is(P)$.
 
Recall from~\cite[Section 4.3]{FOS:2009} that $S(b)$ is invariant under
$\varsigma$, hence in particular $P$ is invariant under $\varsigma$. Also, as operators on
$\hat{B}^{r,s}$,
\begin{equation} \label{Se0}
S(e_0)= \hat{e}_1\hat{e}_0 = \hat e_1 \circ \varsigma \circ \hat e_1 \circ \varsigma
= \hat e_1 \circ \varsigma \circ \hat e_1.
\end{equation}
By~Lemma~\ref{S5.1}, $\hat{e}_1$ either moves a $+$ from $p$ to $P$, or moves a $-$ from $P$
to $p$. That is, $\os(P)$ is unchanged, $\is(P)$ only changes by a single box, and, except for the new $+$ or 
missing $-$, the number of $+$ and $-$ on each row of $P$ is unchanged.
Also, in all situations $\varsigma$ preserves $\is(P)$. Thus~\eqref{Se0} implies that $\is(P)$ can change by at 
most 2 boxes. This along with the description of the classical decomposition of $B^{r,s}$ implies 
part~\eqref{only-dominoes-gC}.

Now assume that $\varepsilon_0(b)>\lceil s/2 \rceil$.
Since $P$ is $\varsigma$-invariant, it is clear that there can be at most $\lfloor s/2 \rfloor$
symbols $-$ in $P$. Since $\varepsilon_0(b) = \hat \varepsilon_1(\hat b) > s/2$, by the description of
$\hat e_1$ from Lemma~\ref{S5.1} we see that $\hat e_1$ must move a $+$ from $p$ to $P$.
There are three cases:
\begin{enumerate}
\item[(a)] The $+$ is added below a $-$ in $P$.
\item[(b)] The $+$ is added in an empty column in $P$ at height less than $r$.
\item[(c)] The $+$ is added in an empty column in $P$ at height $r$.
\end{enumerate}
 
It is not hard to check that in case (a), in $\varsigma\circ \hat{e}_1(P)$, there is one less $+$
in row $k+1$ and one more empty column of height $k-1$ compared to $P$, and everything else is unchanged.
Since both $P$ and  $\hat e_1 \circ \varsigma\circ \hat{e}_1(P)$ are invariant under
$\varsigma$, the only possibility is that the final $\hat e_1$ must add another $+$ to row $k$.
Thus two symbols have been added to $P$ in row $k$, which has the effect of removing
a horizontal domino from $\is(P)$.
 
In case (b), in $\varsigma\circ \hat{e}_1(P)$, there is one less $\mp$-pair
in row $k+1$ and one more $-$ in row $k$ and everything else is unchanged.
Again since  $\hat{e}_1 \circ \varsigma \circ  \hat e_1(b)$ has to be invariant under $\varsigma$, the
only possibility is that $\hat{e}_1$ adds another $+$ to row $k$ of
$\varsigma\circ \hat{e}_1(P)$. Thus $S(e_0)$ removes a horizontal domino from $\is(P)$.

In case (c), $P$ has at least one column with no symbols, and so $P$ contains at most $\lfloor (s-1)/2 \rfloor$
symbols $-$. Since $\hat \varepsilon_1(\hat b)= \varepsilon_0(b) > \lceil s/2 \rceil$, there are at least 2 
uncanceled $+$ in $p$. Consider the 
double $\pm$ diagram $(P',p')$ for $\varsigma \circ \hat e_1(\hat b)$. This differs from $(P,p)$ by: $P'$ has an
extra $-$ at height $r$, and $p'$ has one less $+$. Looking at the cancelation rules before
Lemma~\ref{S5.1}, there is still an uncanceled $+$ in $p'$ (corresponding to the second
uncanceled $+$ in $p$). Thus the second $\hat e_1$ moves another $+$ into $P'$. So in
total $S(e_0) = \hat e_1 \circ \varsigma \circ \hat e_1$ decreases the inner shape of $P$ by 2
boxes, which must remove a horizontal domino.
\end{proof}
 
\subsection{$B^{n,s}$ in type $B_n^{(1)}$ (in case (VIR))} \label{vir2}
 
\begin{theorem}~(see \cite[Lemma 4.2]{FOS:2009}) \label{virtual-KR-th-B}
Let $\hat B^{n,s}$ be the KR crystal of type $A_{2n-1}^{(2)}$ with crystal operators
$\hat e_0, \ldots, \hat e_{n}$.  Let $V$ be the subset of $\hat B^{n,s}$ which
can be reached from the classical highest weight element in $\hat B^{n,s}$ of weight
$s \omega_n$ using the virtual crystal operators $e_i = \hat e_i^2$ for $0 \leq i \leq n-1$,
$e_n = \hat e_n$.
Then $V$ along with the operators $e_0, \ldots, e_n$ is isomorphic to $B^{n,s}$ for
$B_n^{(1)}$.
 
Furthermore, the type $B_n$ highest weight vectors in $V$ are exactly the type $C_n$ highest
weight vectors in $\hat B^{n,s}$ that can be obtained from $\Lambda(s \omega_r)$ by removing
$2 \times 2$ blocks.
\end{theorem}
 
\begin{proof}
The virtual crystal realization is proven in~\cite[Lemma 4.2]{FOS:2009}. The classical
decomposition of $B^{n,s}$ as a type $B_n$ crystal is
\begin{equation*}
        \bigoplus_{\bf{k}} B(k_\iota \omega_\iota +  k_{\iota+2} \omega_{\iota+2}
        + \cdots +  k_{n-2} \omega_{n-2} + k_n \omega_n),
\end{equation*}
where $\iota =0,1$ so that $\iota\equiv n \pmod{2}$, $\omega_0=0$, and the sum is over all
nonnegative integer vectors $\bf k$ such that $2k_\iota+\cdots + 2k_{n-2}+k_n < s$.
As discussed in~\cite[Lemma 4.2]{FOS:2009}, the highest weight vectors for each of these components
must be classical type $C_n$ highest weight vector in the ambient crystal $\hat B^{n,s}$ of weight
\begin{equation*}
        2k_\iota \omega_\iota + 2 k_{\iota+2} \omega_{\iota+2}
        + \cdots + 2 k_{n-2} \omega_{n-2} + k_n \omega_n.
\end{equation*}
The weight of these highest weight vectors in the ambient crystal are exactly those $\gamma$
such that $\Lambda(\gamma)$ is obtained from $\Lambda(s \omega_r)$ by removing
$2 \times 2$ blocks. The result follows.
\end{proof}
 
\begin{lemma}
\label{goode-Bn}
Let $B^{n,s}$ be the KR crystal of type $B_n^{(1)}$.
Fix $b \in B^{n,s}$, and assume $b$ lies in the component $B(\gamma)$ and $e_0(b)$
lies in the classical component $B(\gamma')$. Then
\begin{enumerate}
\item \label{only-squares-Bn}
$\Lambda(\gamma')$ is either equal to $\Lambda(\gamma)$, or else is obtained from
$\Lambda(\gamma)$ by adding or removing a single vertical domino.
\item \label{rest-part-Bn}
If $\varepsilon_0(b)>\lceil s/2 \rceil$,
then $\Lambda(\gamma')$ is obtained from $\Lambda(\gamma)$ by
removing a vertical domino.
\end{enumerate}
\end{lemma}
 
\begin{proof}
For part \eqref{only-squares-Bn}, use the fact that by the virtual realization from
Theorem~\ref{virtual-KR-th-B},  $e_0 = \hat e_0^2$. So, it follows from
Lemma~\ref{goode-D} that $e_0$ adds or subtracts at most 2 vertical dominoes. By
Theorem~\ref{virtual-KR-th-B}, the only possibilities are that $e_0$ either leaves the shape
unchanged, or adds or subtracts a single $2 \times 2$ block.
 
Part \eqref{rest-part-Bn} follows since, if $\varepsilon_0(b)>\lceil s/2 \rceil$, then
$\hat \varepsilon_0(b) \geq s+2$. Hence applying $e_0 = \hat e_0^2$ subtracts 2 vertical
dominoes by Lemma~\ref{goode-D}, which as above must fit together as a $2 \times 2$ block.
\end{proof}
 
\subsection{The KR crystals of type $A_{2n}^{(2)}$ and $D_{n+1}^{(2)}$ in case (VIR)} \label{vir3}
 
We now present the virtual crystal construction of the KR crystals $B^{r,s}$ of types
$A_{2n}^{(2)}$ and $D_{n+1}^{(2)}$ of case (VIR), and then prove the analogue of
Lemma~\ref{goode-D} in this setting.
 
\begin{theorem} \cite[Section 4.4, Theorem 5.7]{FOS:2009} \label{virtual-KR-AD}
Fix $n \geq 2$. Let $\hat{B}^{r,2s}$ be the KR crystal corresponding to type
$C_n^{(1)}$, with crystal operators $\hat e_0, \ldots, \hat e_n$.
In each case below, let $V$ be the subset of $\hat B^{r,2s}$ which can
be reached from the classical highest weight element in $\hat B^{r,2s}$ of weight $2s \omega_r$
using  the listed virtual crystal operators:
\begin{enumerate}
\item  \label{virtual-D2}
Let $1\le r<n$. Define $e_i:= \hat e_i^2$ for $1 \leq i \leq n-1$,
$e_0:= \hat e_0$, and $e_n:= \hat e_n$.  Then $V$, along with the operators
$e_0, \ldots, e_n$, is isomorphic to the KR crystal $B^{r,s}$ for type $D_{n+1}^{(2)}$.
\item  \label{virtual-A2e}
Let $1\le r\le n$. Define $e_i:= \hat e_i^2$ for $1 \leq i \leq n$ and
$e_0:= \hat e_0$. Then $V$ along with the operators $e_0, \ldots, e_n$ is
$B^{r,s}$ for type $A_{2n}^{(2)}$.
\end{enumerate}
Here the operators $f_i$ and $\hat f_i$ are defined by the condition $f_i(b) = b'$ if and only
if $e_i(b')=b$.
\end{theorem}
 
\begin{lemma} \label{goode-A2D2}
Consider $B^{r,s}$ of type $D_{n+1}^{(2)}$ with $r <n$, or of type $A_{2n}^{(2)}$ with $1\le r\le n$.
Fix $b \in B^{r,s}$, and assume $b$ lies in the classical (type $B_n$ or $C_n$, respectively)
component $B(\gamma)$, and $e_0(b)$  lies in the classical component $B(\gamma')$. Then
\begin{enumerate}
\item \label{only-dominoes} $\Lambda(\gamma')$ is either equal to $\Lambda(\gamma)$, or else
is obtained from $\Lambda(\gamma)$ by adding or removing a single box.
\item 
If $\varepsilon_0(b)> s$,
then $\Lambda(\gamma')$ is obtained from $\Lambda(\gamma)$ by
removing a box.
\end{enumerate}
\end{lemma}
 
\begin{proof}
Let $\hat B^{r,2s}$ be the ambient KR crystal of type $C_n^{(1)}$ with crystal operators $\hat e_i$
as in Theorem~\ref{virtual-KR-AD}. Denote the virtualization map $S: B^{r,s} \rightarrow V
\subseteq \hat B^{r,2s}$. Recall that $V$ is the subset of $\hat B^{r,2s}$ generated by
the element of weight $2s \omega_r$ by applying the virtual crystal operators $\hat e_i$:
\begin{enumerate}
\item In type $D_{n+1}^{(2)},$ $e_0 = \hat e_0,$ $e_n = \hat e_n$, and $e_i = \hat e_i^2$
for $1 \leq i \leq n-1$.
\item  In type $A_{2n}^{(2)},$ $e_0 = \hat e_0,$ and $e_i = \hat e_i^2$ for $1 \leq i \leq n$.
\end{enumerate}
By Lemma~\ref{goode-C}, $\hat e_0$ changes the classical component of the underlying KR
crystal of type $C_n^{(1)}$ by adding or subtracting at most one horizontal domino, and furthermore
if $\hat \varepsilon_0(S(b)) > s$, then $\hat e_0$ always removes a horizontal domino. The result
now follows immediately from the description of the classical components of these virtual crystals
as given in~\cite[Lemma 4.9]{FOS:2009}.
 
Note that in the case of $A_{2n}^{(2)}$ with $r=n$, $\hat B^{n,2s}$ is not a KR crystal of type
$C_n^{(1)}$. However, Theorem~\ref{virtual-KR-C} still gives a type $C_n^{(1)}$ combinatorial
crystal in this case and Lemma~\ref{goode-C} still holds.
So the proof still goes though in this case.
\end{proof}

\section{Energy functions}
\label{section.energy}
 
We define two energy functions on tensor products of KR crystals. The
function $E^\intrinsic$ is given by a fairly natural ``global" condition on tensor products of level-$\ell$ KR crystals. 
The function $D$ is defined by summing up combinatorially defined ``local" contributions, and makes sense for 
general tensor products of KR crystals. It was suggested  (but not proven) in~\cite[Section 2.5]{OSS:2002} that, 
when $E^\intrinsic$ is defined, these two functions agree up to a shift. This will be proven in
Theorem~\ref{E=D} below.
 
\subsection{The function $E^\intrinsic$} \label{Eint:section}
 
\begin{definition} \label{def:u tau}
For each node $r\in I\setminus\{0\}$ and each $\ell \in \bz_{>0}$, let $u_{r, \ell c_r}$ be the
unique element of $B^{r,\ell c_r}$ such that $\varepsilon(u_{r, \ell c_r})=\ell \Lambda_0$
(which exists as $B^{r, \ell c_r}$ is perfect).
\end{definition}
 
The following is essentially the definition of a ground state path from \cite{KMN1:1992}.
 
\begin{definition} \label{def:uB}
Let $B= B^{r_N, \ell c_{r_N}} \otimes \cdots \otimes B^{r_1, \ell c_{r_1}}$ be a composite level-$\ell$ 
KR crystal. Define $u_B = u_B^N \otimes \cdots \otimes u_B^1$ to be the unique element of $B$ such that
\begin{enumerate}
\item $u_B^1 = u_{r_1,\ell c_{r_1}}$ and
\item for each $1 \leq j < N$, $\varepsilon(u_B^{j+1}) = \varphi(u_B^j)$.
\end{enumerate}
This $u_B$ is well-defined by condition~\eqref{perfect:bij} in Definition~\ref{def:perfect} of a perfect crystal.
The element $u_B$ is called the \textit{ground state path} of $B$.
\end{definition}
 
\begin{definition} \label{def:intrinsic}
Let $B$ be a composite KR crystal of level $\ell$ and consider $u_B$ as in Definition~\ref{def:uB}. The
\textit{intrinsic energy function} $E^{\intrinsic}$ on $B$ is defined by setting $E^{\intrinsic}(b)$ to be the minimal
number of $f_0$ in a string $f_{i_N} \cdots f_{i_1}$ such that  $f_{i_N} \cdots f_{i_1}(u_B)=b$.
\end{definition}
 
\subsection{The $D$ function} \label{D-section}
 
\begin{definition} \label{primeD-def}
The $D$-function on $B^{r,s}$ is the function defined as follows:
\begin{enumerate}
\item $D_{B^{r,s}} : B^{r,s} \to \Z$ is constant on all classical components.
\item On the component $B(\lambda)$, $D_{B^{r,s}}$ records the maximum number of $\diamond$
that can be removed from $\Lambda(\lambda)$ such that the result is still a (generalized) partition,
where $\diamond$ is as in \eqref{diamond-eq}.
\end{enumerate}
In those cases when $\diamond = \emptyset$, this is interpreted as saying that $D_{B^{r,s}}$ is the constant function $0$.
\end{definition}
 
Let $B_1$, $B_2$ be two affine crystals with generators $v_1$ and $v_2$, respectively, such
that $B_1\otimes B_2$ is connected and $v_1\otimes v_2$ lies in a one-dimensional weight space.
By~\cite[Proposition 3.8]{LOS:2010}, this holds for any two
KR crystals. The generator $v$ for the KR crystal $B^{r,s}$ is chosen to be the unique element of classical weight 
$s\omega_r$.
 
The \textit{combinatorial $R$-matrix}~\cite[Section 4]{KMN1:1992} is the unique
crystal isomorphism
\begin{equation*}
    \sigma : B_2 \otimes B_1 \to B_1 \otimes B_2.
\end{equation*}
By weight considerations, this must satisfy $\sigma(v_2 \otimes v_1) = v_1 \otimes v_2$.
 
As in~\cite{KMN1:1992} and~\cite[Theorem 2.4]{OSS:2003}, there is a function
$H=H_{B_2,B_1}:B_2\otimes B_1\rightarrow\Z$, unique up to global additive constant, such that,
for all $b_2\in B_2$ and $b_1\in B_1$,
\begin{equation} \label{eq:local energy}
 H(e_i(b_2\otimes b_1))=
 H(b_2\otimes b_1)+
 \begin{cases}
   -1 & \text{if $i=0$ and LL,}\\
   1 & \text{if $i=0$ and RR,}\\
   0 & \text{otherwise.}
 \end{cases}
\end{equation}
Here LL (resp. RR) indicates that $e_0$ acts on the left (resp. right) tensor factor in both
$b_2\otimes b_1$ and $\sigma (b_2 \otimes b_1)$. When $B_1$ and $B_2$ are KR crystals, we normalize
$H_{B_2, B_1}$ by requiring $H_{B_2, B_1}(v_2 \otimes v_1)= 0$, where $v_1$ and $v_2$ are the generators
defined above.
 
\begin{definition} \label{DB-def}
For $B=B^{r_N,s_N} \otimes \cdots \otimes B^{r_1,s_1}$, $ 1 \leq i \le N$ and $i< j \leq N,$ set
\begin{equation*}
        D_i :=  D_{B^{r_i,s_i}}\sigma_1 \sigma_2 \cdots \sigma_{i-1} \quad \text{and} \quad
        H_{j,i}:= H_{i} \sigma_{i+1} \sigma_{i+2} \cdots \sigma_{j-1},
\end{equation*}
where $\sigma_i$ and $H_i$ act on the $i$-th and $(i+1)$-st tensor factors and
$D_{B^{r_i,s_i}}$ is the $D$-function on the rightmost tensor factor $B^{r_i,s_i}$ as given in Definition \ref{primeD-def}. 
The $D$-function $D_B : B \to \Z$ is defined as
\begin{equation}
       D_B := \sum_{N \geq j > i \geq 1} H_{j,i} + \sum_{i=1}^N D_i.
\end{equation}
Where it does not cause confusion, we shorten $D_B$ to simply $D$.
\end{definition}

\section{Perfect KR crystals and Demazure crystals}
\label{section.demazure}
 
We now state a precise relationship between KR crystals and Demazure
crystals (see Theorem~\ref{FSS-theorem}). This was proven  by Fourier, Schilling, and
Shimozono~\cite{FSS:2007}, under a few additional assumptions on the KR crystals
since at the time the existence and combinatorial models for KR crystals did not yet exist.
Here we point out that most of the assumptions in~\cite{FSS:2007} follow 
from the later results of~\cite{Okado:2007,OS:2008,FOS:2009,FOS:2010} showing that the relevant KR crystals 
exist, have certain symmetries related to Dynkin diagram automorphisms, and are perfect. 
In the special cases of type $A_{2n}^{(2)}$ and the exceptional nodes in type $D_n^{(1)}$,
we give separate proofs as the assumptions from~\cite{FSS:2007} do not follow directly from the above
papers or need to be slightly modified.

\begin{theorem} \label{FSS-theorem}
Let  $B= B^{r_N, \ell c_{r_N}} \otimes \cdots \otimes B^{r_1, \ell c_{r_1}}$ be a level-$\ell$ composite KR crystal
of non-exceptional type. Define $\lambda = -({c_{r_1} \omega_{r^*_1}}+ \cdots + {c_{r_N} \omega_{r_N^*}})$,
where $r^*$ is defined by $\omega_{r^*}=-w_0(\omega_r)$ with $w_0$ the longest element of $\overline{W}$, 
and write $t_\lambda \in T(\widetilde M) \subset \widetilde W$
as $t_\lambda  = v \tau$. Then there is a unique isomorphism of affine crystals
\begin{equation*}
       j: B(\ell \Lambda_{\tau(0)})  \rightarrow B \otimes B(\ell \Lambda_0).
\end{equation*}
This satisfies
\begin{equation*}
	j(u_{\ell \Lambda_{\tau(0)}})= u_B \otimes u_{\ell \Lambda_0},
\end{equation*}
where $u_B$ is the distinguished element from Definition~\ref{def:uB}, and
\begin{equation} \label{DRK}
       j \left(  B_v(\ell \Lambda_{\tau(0)}) \right)
       = B \otimes u_{\ell \Lambda_0},
\end{equation}
where $B_v(\ell \Lambda_{\tau(0)})$ is the Demazure crystal corresponding to the translation $t_\lambda$
as defined in Section~\ref{DC-sec}.
\end{theorem}
 
We delay the proof of Theorem~\ref{FSS-theorem} until the end of this section.

\begin{remark}
For non-exceptional types, we have $r^*=r$ except for type $A_n^{(1)}$ where $r^*=n+1-r$, and
$D_n^{(1)}$ for $n$ odd where $n^*=n-1$ and $(n-1)^*=n$.
\end{remark}
 
\begin{lemma}
\label{A:KR}
Assume $\g$ is of non-exceptional type, and let $B^{r,\ell c_r}$ be a level-$\ell$ KR crystal. Then:
\begin{enumerate}
\item \label{A:u}
There is a unique element $u\in B^{r,\ell c_r}$ such that
\begin{align*}
        \varepsilon(u)=\ell \Lambda_0 \quad \text{and} \quad \varphi(u) = \ell \Lambda_{\tau(0)},
\end{align*}
where $t_{-c_r\omega_{r^*}} = v \tau$ with $v\in W$ and $\tau\in \Sigma$.
\item \label{A:auto}
Let $\varsigma$ be the Dynkin diagram automorphism defined by
\begin{itemize}
\item  For type $A_n^{(1)}$, $i\mapsto i+1 \pmod{n+1}$.
\item  For types $B_n^{(1)}$, $D_n^{(1)}$, $A_{2n-1}^{(2)}$, exchange nodes
0 and 1.
\item For types $C_n^{(1)}$ and $D_{n+1}^{(2)}$,  $i\mapsto n-i$ for all $i\in I$.
\end{itemize}
In all cases other than type $A_{2n}^{(2)}$ and  
$B^{n-1,s}$ and $B^{n,s}$ for type $D_n^{(1)}$, there
is a unique involution of $B^{r,\ell c_r}$, also denoted $\varsigma$, such that for all
$b \in B^{r,\ell c_r}$ and $i \in I$,
\begin{equation*}
f_i(b) = \varsigma^{-1} \circ  f_{\varsigma(i)} \circ \varsigma(b).
\end{equation*}
\end{enumerate}
\end{lemma}
 
\begin{remark}
Lemma~\ref{A:KR} was stated as~\cite[Assumption 1]{FSS:2007} (with the misprint $t_{-c_r \omega_r}$ instead
of $t_{-c_r \omega_{r^*}}$ in~\eqref{A:u}). This Assumption also included
the requirement that $B^{r,\ell c_r}$ is regular. Since we now know that $B^{r,\ell c_r}$ is the crystal
of a KR module by~\cite{Okado:2007,OS:2008,FOS:2009}, this is immediate, and so we do not
include it in Lemma~\ref{A:KR}. The statement in~\cite{FSS:2007} also included an additional assumption
for type $A_{2n}^{(2)}$, which we omit as we deal with type $A_{2n}^{(2)}$ separately.
\end{remark}
 
\begin{proof}[Proof of Lemma \ref{A:KR}]
By perfectness of $B^{r,\ell c_r}$ for non-exceptional types~\cite{FOS:2010}, there exists a
unique element $u\in B^{r,\ell c_r}$ such that $\varepsilon(u)=\ell \Lambda_0$. These elements
are listed for all non-exceptional types in Figure \ref{u-table}. It is easy to check explicitly that
these also satisfy $\varphi(u) = \ell \Lambda_{\tau(0)}$, where $\tau$ is as listed. Furthermore,
one can easily check that all these $\tau$ satisfy the remaining conditions of~\eqref{A:u}.
 
\begin{figure}
\begin{equation*}
\begin{array}{|l|l|c|c|l|}
        \hline
        \text{Type} & u \in B^{r,\ell c_r} & r & c_r & \tau \\ \hline \hline
        A_n^{(1)}   & u(\ell\Lambda_r) & 1\le r \le n & 1& \mathrm{pr}^r \\ \hline
        B_n^{(1)}  &  u(\emptyset) & 1\le r <n, \text{$r$ even} & 1& \mathrm{id}\\
                            &  u(\ell \Lambda_1)   & 1\le r < n, \text{$r$ odd} & 1 & \varsigma_{0,1}\\
                                &  u(\emptyset) & \text{$r=n$ even} & 2& \mathrm{id}\\
                            & u(\ell \Lambda_1) & \text{$r=n$ odd} & 2 & \varsigma_{0,1}\\ \hline
          C_n^{(1)} & u(\emptyset) & 1\le r< n & 2 & \mathrm{id}\\
                             & u(\ell \Lambda_n) & r=n & 1 & \varsigma_{\leftrightarrow}\\ \hline
          D_n^{(1)} & u(\emptyset) & 1\le r\le n-2, \text{$r$ even} & 1& \mathrm{id}\\
                             & u(\ell\Lambda_1) & 1\le r\le n-2, \text{$r$ odd} & 1& \varsigma_{0,1} \varsigma_{n-1,n}\\
                             & u(\ell \Lambda_{n-1}) & r=n-1& 1& \varsigma_{\leftrightarrow}\varsigma_{0,1}\varsigma_{n-1,n}^{n+1}\\
                             & u(\ell \Lambda_n) & r=n& 1& \varsigma_{\leftrightarrow} \varsigma_{n-1,n}^n\\ \hline
          A_{2n-1}^{(2)} & u(\emptyset) & 1\le r\le n, \text{$r$ even} & 1 & \mathrm{id}\\
                                     & u(\ell \Lambda_1) & 1\le r \le n, \text{$r$ odd} & 1 & \varsigma_{0,1}\\ \hline
          D_{n+1}^{(2)} & u(\emptyset) & 1\le r<n & 1 & \mathrm{id}\\
                                    & u(\ell \Lambda_n) & r=n & 1 & \varsigma_{\leftrightarrow}\\ \hline
          A_{2n}^{(2)} & u(\emptyset) & 1\le r\le n & 1& \mathrm{id}\\ \hline
\end{array}
\end{equation*}
 
\caption{The elements $u$ and related data. Here $u(\lambda)$ is the highest weight vector in the classical
component $B(\lambda)$, $\mathrm{pr}$ is the map $i\mapsto i+1 \pmod{n+1}$, and $\varsigma_{0,1}$
(resp. $\varsigma_{n-1,n}$) is the map that interchanges 0 and 1 (resp. $n-1$ and $n$) and
fixes all other $i$. The map $\varsigma_{\leftrightarrow}$ is $i \mapsto n-i$. In the expression for 
$\tau$ the maps act on the left, so that e.g. $\varsigma_\leftrightarrow \varsigma_{0,1} (1) 
= \varsigma_\leftrightarrow(0)= n$. 
\label{u-table}}
\end{figure}
 
The combinatorial model for $B^{r,s}$ of type $A_n^{(1)}$ uses the promotion operator
which corresponds to the Dynkin diagram automorphism mapping $i\mapsto i+1 \pmod{n+1}$.
Similarly, the KR crystals in case (AUT) are constructed using the automorphism $\varsigma$
(see Section~\ref{caseAUT:section}).

By Theorem~\ref{virtual-KR-th-B} the KR crystal $B^{n,2\ell}$ of type $B_n^{(1)}$
can be constructed as a virtual crystal inside $\hat B^{n,2\ell}$ of type $A_{2n-1}^{(2)}$ with
$e_0=\hat e_0^2$ and $e_1=\hat e_1^2$. The virtual crystal is constructed using the
analogue of the Dynkin automorphism $\hat \varsigma$ exchanging $0$ and $1$, so
that $e_0 = \hat e_0^2 = (\hat \varsigma \hat e_1 \hat \varsigma) (\hat \varsigma \hat e_1\hat
\varsigma) = \hat \varsigma \hat e_1^2 \hat \varsigma = \hat \varsigma e_1 \hat \varsigma$.
Furthermore, by~\cite[Lemma 3.5]{FOS:2009} the image of $B^{n,2\ell}$ in the ambient crystals
$\hat{B}^{n,2\ell}$ is closed under $\hat{\varsigma}$.
Hence on $B^{n,2\ell}$ of type $B_n^{(1)}$ there also exists an isomorphism interchanging
nodes 0 and 1, induced by $\hat{\varsigma}$.

It was shown in~\cite[Theorem 7.1]{FOS:2009} that $B^{r,s}$ of types $C_n^{(1)}$ and
$D_{n+1}^{(2)}$ admit a twisted isomorphism $\varsigma$ corresponding to the Dynkin
automorphism mapping $i\mapsto n-i$ for $i\in I$. This shows~\eqref{A:auto}.
 \end{proof}
 
For $B$ a crystal of affine type and $\tau \in \Aut(\Gamma)$, where $\Gamma$ is the affine
Dynkin diagram, let $B^\tau$ be the crystal with the same underlying set as $B$, but where
$e_i^\tau = \tau \circ e_i \circ \tau^{-1}$ and $f_i^\tau = \tau \circ f_i \circ \tau^{-1}$.
 
\begin{lemma} \label{D-extra-auts} 
Let $\Gamma$ be a non-exceptional affine Dynkin diagram and $\tau \in \Sigma$.
Then $(B^{r,s})^\tau \cong B^{r,s}$ for all $r\in I\setminus \{0\}$ and $s\ge 1$.
\end{lemma}
 
\begin{proof}
Fix $\tau \in \Sigma$. By Lemma~\ref{Sigma-alt}, there exists $w_\tau \in \oW$ such that
  \begin{equation}
  \xymatrix{
    I \ar[r]^\iota \ar[d]^\tau & \Delta \ar[d]^{w_\tau} \\
    I \ar[r]_{\iota} & \Delta
  }
\end{equation}
commutes, where $\iota$ is the map that takes $i$ to $\alpha_i$ for all $i \neq 0$, and $0$ to
$- \theta$ (where $\theta = \delta-\alpha_0$). Thus the character of $(B^{r,s})^\tau$ as a
$\{1,2, \ldots, n\}$-crystal is in the Weyl group orbit of the character of $B^{r,s}$ as a
$\{1,2,\ldots, n\}$-crystal, and hence by Weyl group invariance these are equal. By the classification
of classical crystals, it follows that $(B^{r,s})^\tau \cong B^{r,s}$ as a  $\{1,2, \ldots, n\}$-crystal.
Thus the lemma follows by Lemma~\ref{uniqueness}.  
\end{proof}
 
\begin{lemma} \label{T4.7-D}
Theorem~\ref{FSS-theorem} holds in type $D_n^{(1)}$.
\end{lemma}
 
\begin{proof}
For $B=B^{r,\ell}$ with $1\le r \leq n-2$, Lemma~\ref{A:KR} shows that \cite[Assumption 1]{FSS:2007}
holds, and so Theorem~\ref{FSS-theorem} holds by~\cite[Theorem 4.4]{FSS:2007}.
 
Now consider $B^{n, \ell}$, noting that $c_n=1$. We need the following notation:

$\bullet$ As in Theorem \ref{FSS-theorem}, $n^*$ is defined by $\omega_{n^*} = -w_0(\omega_n)$, 
so that $(w_0^{n^*})^{-1}=w_0^n$. 
 
$\bullet$ $w_0^n \in \oW$ is of minimal length such that $w_0^n(\omega_n)= w_0(\omega_n)$,  
 
$\bullet$
$w_0^{n^*}$ is of minimal length such that
$w_0^{n^*}(\omega_{n^*}) = w_0(\omega_{n^*}),$ 
 
 $\bullet$ $\tau_n$ and $\tau_{n^*}$ are the $\tau$ from Table \ref{u-table} for $r=n, n^*$ respectively.

\noindent As in \cite[Equation (2.10)]{FSS:2007} (but noting that  due to differing conventions their $\tau_{n^*}$ 
is our $\tau_{n}$), 
\begin{equation}
t_{-\omega_{n^*}}= (w_0^{n^*})^{-1} \tau_{n} = w_0^{n} \tau_{n}.
\end{equation}
Since $w_0^n(\ell \Lambda_n)= w_0(\ell\Lambda_n)$, we see that
\begin{equation} \label{D-fs}
        B_{w_0^n}(\ell\Lambda_n)= B_{w_0}(\ell\Lambda_n)= \{  f_{i_N}^{m_N} \cdots f_{i_1}^{m_1} u_{\ell \Lambda_n}
        \mid m_1, \ldots, m_N \geq 0 \},
\end{equation}
where $s_{i_1} \cdots s_{i_N}$ is a reduced word for $w_0$ and the last equality follows 
by~\cite[Proposition 3.2.3]{Kashiwara:1993} (or see Section~\ref{DC-sec} above).
 
As in Section~\ref{class-decomp}, $B^{n,\ell}$ is isomorphic to $B(\ell \Lambda_n)$ as a classical crystal. 
The element $u$ from Figure~\ref{u-table} is the highest weight element $u(\ell \Lambda_n)$ of weight 
$\ell \Lambda_n$. It satisfies $\varepsilon(u(\ell \Lambda_n))=\ell \Lambda_0$ and 
$\varphi(u(\ell \Lambda_n))=\ell \Lambda_n$.
Furthermore $B^{n,\ell}$ is perfect of level $\ell$ so, as in~\cite{KMN1:1992}, there is a unique isomorphism
\begin{equation}\label{j-Dn}
\begin{aligned}
        j: B(\ell \Lambda_n) & \rightarrow B^{n,\ell} \otimes B(\ell\Lambda_0) \\
        u_{\ell \Lambda_n} & \mapsto u(\ell \Lambda_n) \otimes u_{\ell\Lambda_0}.
\end{aligned}
\end{equation}
Since $\varepsilon_i(u_{\ell \Lambda_0}) = \varphi_i(u_{\ell \Lambda_0}) =0$ for all $i \neq 0$, 
\eqref{D-fs} and~\eqref{j-Dn} imply
\begin{equation}
\begin{aligned} \label{j-DnD}
       j \left( B_{w_0^n}(\ell \Lambda_n) \right)
       = B^{n,\ell} \otimes u_{\ell \Lambda_0}.
\end{aligned}
\end{equation}
Since $\tau_{n}(0)=n$, we have proven Theorem \ref{FSS-theorem} 
for $B=B^{n,s}$. A similar argument shows that Theorem~\ref{FSS-theorem} holds for $B=B^{n-1,s}$.
 
By Lemma~\ref{D-extra-auts}, each automorphism $\tau \in \Sigma$
induces a bijection of $B^{r,\ell}$ to itself that sends $i$ arrows to $\tau(i)$ arrows. Thus, as
in~\cite[Theorem 4.7]{FSS:2007}, a straightforward induction argument shows that
Theorem~\ref{FSS-theorem} holds for all composite level-$\ell$ KR crystals.
\end{proof}

\begin{lemma} \label{T4.7-A}
Theorem \ref{FSS-theorem} holds in type $A_{2n}^{(2)}$.
\end{lemma}
 
\begin{proof} 
In~\cite{FSS:2007} a proof of Theorem~\ref{FSS-theorem} was given, but not with the concrete
combinatorial model given in Theorem~\ref{virtual-KR-AD}.
The only place, where the concrete combinatorial model for type $A_{2n}^{(2)}$ is used
in~\cite{FSS:2007} is in the proof of~\cite[Lemma 4.3]{FSS:2007}. Thus we must verify that this
lemma holds with the explicit realization of KR crystals for this type from \cite{FOS:2009} or
equivalently Theorem~\ref{virtual-KR-AD}.
Let $S: B^{r,s} \to \hat B^{r,2s}_{C_n^{(1)}}$ be the unique injective map from the combinatorial
KR crystal of type $A_{2n}^{(2)}$ to the KR crystal of type $C_n^{(1)}$ such that
\begin{equation*}
        S(e_i b) = \hat e_i^{m_i} S(b) \qquad \text{and} \qquad
        S(f_i b) = \hat f_i^{m_i} S(b)    \quad \text{for all $i\in I$ and $b\in B^{r,s}$,}
\end{equation*}
where $m=(m_0,\ldots,m_n)=(1,2,\ldots,2)$. Using the explicit realization of
$\hat B^{r,2s}_{C_n^{(1)}}$ it is not too hard to check that for the classically highest weight
element $u(s\omega_k) \in B(s\omega_k) \subset B^{r,s}$ for $k<r$ and
$y:=S_1 \cdots S_k(u(s\omega_k))$ (where $S_i$ is $f_i$ raised to the maximal power), we have
$        S(f_0^s y) = S(u(s\omega_{k+1})).$
Thus~\cite[Lemma 4.3]{FSS:2007} holds. Furthermore, in this case $\tau$ is always the identity.
Thus Theorem~\ref{FSS-theorem} holds by~\cite[Theorem 4.7]{FSS:2007}.
\end{proof}

\begin{proof}[Proof of Theorem \ref{FSS-theorem}]
If $X=D_n^{(1)}$ or $A_{2n}^{(2)}$, the Theorem holds by Lemmas \ref{T4.7-D} and \ref{T4.7-A}.
In all other cases, Lemma~\ref{A:KR} shows that~\cite[Assumption 1]{FSS:2007} holds. Furthermore,
in all these cases $\Sigma$ is generated by the diagram automorphism $\varsigma$ from
Lemma~\ref{A:KR}. Thus the theorem holds by \cite[Theorem 4.7]{FSS:2007}.
\end{proof}

\section{The affine grading via the energy function}
\label{section.energy.equiv}
 
In this section, we show that for $B = B^{r_N, \ell c_{r_N}} \otimes \cdots \otimes B^{r_1, \ell c_{r_1}}$ 
a composite level-$\ell$ KR crystal the map $j$ from Theorem~\ref{FSS-theorem}
intertwines the $D$ function from Section~\ref{D-section} with the affine degree $\deg$ given in 
Definition~\ref{deg-def} below up to a shift.
This allows us to show that $j$ intertwines $E^\intrinsic$ with $\deg$ exactly, 
and in particular $E^\intrinsic$ agrees with $D$ up to a shift. 

\begin{definition} \label{deg-def}
For any Demazure crystal $B_w(\Lambda)$, where $w \in W$ and $\Lambda$ is a dominant integral weight, let
\[
       \deg:  B_{w}(\Lambda) \rightarrow \bz_{\geq 0}
\]
be the affine degree map defined by $\deg(u_{\Lambda})= 0$ and each $f_i$ has degree $\delta_{i,0}.$
\end{definition}
 
We begin by two preliminary lemmas, first for a single KR crystal and then tensor products of KR crystals.
 
\begin{lemma} \label{prime-contribution}
Fix a KR crystal $B^{r,s}$ of type $\g$. Then $D(e_0(b)) \geq  D(b) -1$ for all $b\in B^{r,s}$.
Furthermore, if $\varepsilon_0(b)>\lceil s/c_r \rceil$, then we have  $D(e_0(b)) = D(b) -1$.
\end{lemma}
 
\begin{proof} \mbox{}
\begin{itemize}
\item \textbf{Case (IRR):}
In this case $\diamond=\emptyset$, so that by Definition~\ref{primeD-def} the $D$ function is the constant
function $0$ which satisfies $D(e_0(b)) \geq  D(b) -1$ for all $b\in B^{r,s}$.
By Lemma~\ref{lemma:irreducible} we always have $\varepsilon_0(b)\le s/c_r$ since $c_r=1$.
\item \textbf{Case (AUT):}
In this case $\diamond$ is a vertical domino. The statements follow
immediately from Lemma~\ref{goode-D} since again $c_r=1$ in all cases.
\item
For $B^{r,s}$ of type $C_n^{(1)}$ with $r<n$, we have $c_r=2$ and $\diamond$ is a horizontal domino.
The result follows from Lemma~\ref{goode-C}.
\item
For $B^{n,s}$ of type $B_n^{(1)}$, we have $c_r=2$ and $\diamond$ is a vertical domino.
The result follows from Lemma \ref{goode-Bn}.
\item
For $B^{r,s}$ of types $A_{2n}^{(2)}$ with $r\le n$ and $D_{n+1}^{(2)}$ with $r \leq n-1$, 
we have $c_r=1$ and $\diamond$ is a single box. The result follows from Lemma~\ref{goode-A2D2}.
\end{itemize}
\end{proof}
 
\begin{lemma} \label{DD-on-B} 
Let $B= B^{r_N, s_N} \otimes \cdots \otimes  B^{r_1, s_1}$ be a tensor product of KR crystals and fix an integer $\ell$
such that $\ell \geq \lceil s_k/c_k \rceil$ for all $1 \leq k \leq N$. If $e_0(b) \neq 0$ then $D(e_0(b)) \geq D(b)-1$,
and if $\varepsilon_0(b) > \ell$ then this is an equality.
\end{lemma}
 
\begin{proof}
Write $b= b_N \otimes \cdots \otimes b_1.$ For some $N \ge k \ge 1$,
\[
e_0(b) = b_N \otimes \cdots \otimes e_0(b_k) \otimes \cdots\otimes b_1.
\]
Let
\begin{equation*}
        b'_N \otimes \cdots \otimes b'_1 = \sigma_1 \cdots \sigma_{k-1}(b).
\end{equation*}
Since each $\sigma_j$ commutes with the action of $e_0$,
\[
e_0(b'_N \otimes \cdots \otimes b'_1)= b'_N \otimes \cdots \otimes e_0(b'_c) \otimes \cdots \otimes b'_1
\]
for some $k \geq c \geq 1$.
For each $ 1\le x\le N$, let
\begin{equation}
       D^{[x]} = \sum_{1\le y < x} H_{x,y} + D_x,
\end{equation}
where $H_{x,y}$ and $D_x$ are as in Definition \ref{DB-def}, so that
$D = \sum_{1\le x\le N} D^{[x]}$.
Let $\Delta(b)= D(b) - D(e_0(b))$ and $\Delta^{[x]}(b) = D^{[x]}(b) - D^{[x]}(e_0(b))$. We find each $\Delta^{[x]}(b)$, considering three cases.
 
{\bf Case 1, $k >x$:} Here $\Delta^{[x]}(b)=0$, as all terms in $D^{[x]}(b)$ and $D^{[x]}(e_0(b))$ agree.  
 
{\bf Case 2, $x=k$:} If $c \neq 1$, then, by Definitions~\ref{primeD-def} and~\ref{DB-def},
$D_x(b)=D_x(e_0(b))$ and by \eqref{eq:local energy}, $H_{x,c-1}( e_0(b)) = H_{x,c-1}( b)-1$.
For all other $y<x$  it is clear that $H_{x,y}(b)= H_{x,y}(e_0(b))$. Hence $\Delta^{[x]}(b)=1$.
 
Otherwise $c =1$ and
\[
       \Delta^{[x]}(b)= D^{[x]}(b) - D^{[x]}(e_0(b))
        = D_{B^{r_1,s_1}}(b'_1)- D_{B^{r_1, s_1}}(e_0(b_1')).
\]
In this case $\varepsilon(b'_1)= \varepsilon(b)$, so by Lemma \ref{prime-contribution}, $\Delta^{[x]}(b) \leq 1$, with equality if $\varepsilon_0(b) >\ell \geq s_{r_1}/c_{r_1}$.
 
{\bf Case 3, $x>k$}:
It is clear that $H_{x,y}(b) = H_{x,y}(e_0(b))$ for $x>y>k$.
Let
\begin{equation*}
d= d_N \otimes \cdots \otimes d_1 := \sigma_k \cdots \sigma_{x-1} (b ).
\end{equation*}
Since each $\sigma_j$ commutes with the action of $e_0$, one of the following must hold:
\begin{enumerate}
\item[(a)] $e_0(d_N \otimes \cdots \otimes d_1) = d_N \otimes \cdots \otimes e_0(d_{k+1}) \otimes d_k \otimes \cdots \otimes d_1$,
\item[(b)] $e_0(d_N \otimes \cdots \otimes d_1) = d_N \otimes \cdots \otimes d_{k+1} \otimes e_0(d_k) \otimes \cdots \otimes d_1.$
\end{enumerate}
If (a) holds, then for $y<k$ we have $H_{x,y}(b) = H_{x,y}(e_0(b))$, since $H_{x,y}$ is calculated
on exactly the same tensor product on each side. By \eqref{eq:local energy} we have
$H_{x,k}(b) = H_{x,k}(e_0(b))$ and by Definition \ref{primeD-def} $D_x(b)= D_x(e_0(b))$. Hence $\Delta^{[x]}(b)=0$.
 
If (b) holds, then, by~\eqref{eq:local energy}, $H_{x,k}(e_0(b)) = H_{x,k}(b)+1$, and for all $k < y < x$,
$H_{x,y}(e_0(b)) = H_{x,y}(b)$. Furthermore, by Definition \ref{DB-def} $D_x(b)= D_k(d)$ and $D_x(e_0 b)= D_k(e_0 d)$.
Comparing terms and noticing that $H_{x,r}(b)= H_{k,r}(d)$ and $H_{x,r}(e_0 b)= H_{k,r}(e_0 d)$ for all $r<x$ shows
\[
       \Delta^{[x]}(b) = \Delta^{[k]}(d) -1,
\]
so it follows by the argument in Case 2 that $\Delta^{[x]}(b) \leq 0$, with equality if
$\varepsilon_0(b) =\varepsilon_0(d) >\ell$.
 
The result now follows by adding all the $\Delta^{[x]}(b)$.
\end{proof}
 
\begin{theorem}  \label{grading=energy}
With the same assumptions and notation as in Theorem~\ref{FSS-theorem},
let $\tilde j: B_v(\ell \Lambda_{\tau(0)}) \rightarrow B$ be the restriction of the map $j$ to
$B_v(\ell \Lambda_{\tau(0)})$, where $B \otimes u_{\ell \Lambda_{0}}$ is identified with just $B$.
Then for all $b \in B_v(\ell \Lambda_{\tau(0)})$ we have $\deg (b) = D(\tilde j(b)) -D(u_B)$.
\end{theorem}
 
\begin{proof}
Since $B_v(\ell\Lambda_{\tau(0)})$ is connected with the highest weight element $u_{\ell \Lambda_{\tau(0)}}$,
$\tilde j(u_B) = u_{\ell \Lambda_{\tau(0)}}$, and $\deg(u_{\ell \Lambda_{\tau(0)}})=0$, it suffices to show that
$D(\tilde j(b)) = D(\tilde j (e_0(b)))-1$ for each $b \in B_v(\ell\Lambda_{\tau(0)})$ such that $e_0(b ) \neq 0$.
So, choose such a $b \in B_v(\ell\Lambda_{\tau(0)})$.
Since $e_0$ acts non-trivially on $\tilde j(b) \otimes u_{\ell \Lambda_0}$, we have $\varepsilon_0(\tilde j(b)) > \ell$,
and the result follows by Lemma~\ref{DD-on-B}.
\end{proof}
 
\begin{theorem} \label{E=D}
For any composite level-$\ell$ KR crystal $B=B^{r_N, \ell c_{r_N}} \otimes \cdots \otimes
B^{r_1, \ell c_{r_1}}$ we have $E^\intrinsic =D -D(u_B)$.
\end{theorem}
 
\begin{proof}
By the definition of $E^\intrinsic$ (see Section \ref{Eint:section}), for any $b$ there exists a string of $f_i$ involving
exactly $E^\intrinsic(b)$ factors $f_0$ taking $u_B$ to $b$. By Lemma~\ref{DD-on-B}, applying $f_0$ increases
the value of $D$ by at most $1$, and so it follows that $E^\intrinsic(b) \geq D(b) - D(u_B)$.
It follows from Theorem \ref{grading=energy} and the fact that the Demazure crystal  
$B_v(\ell \Lambda_{\tau(0)})$ 
is connected that $E^\intrinsic(b) \leq D(b)- D(u_B)$.
\end{proof}
 
\begin{corollary} \label{Edeg}
The map $j$ intertwines $E^\intrinsic$ and the basic grading $\deg$.
\end{corollary}
 
\begin{proof}
This is immediate from Theorems \ref{grading=energy} and  \ref{E=D}.
\end{proof}

\begin{remark} \label{remark.demazure arrow}
Let $B$ be a composite level-$\ell$ KR crystal. If $\varepsilon_0(b) > \ell$, call the corresponding $0$ arrow from 
$b'= e_0(b)$ to $b$ in $B$ a \textit{Demazure arrow}. We have shown that for such an arrow $D(b)= D(b')+1$. 
Furthermore, we have shown that $u_B$ is the unique source in the subgraph of $B$ consisting only of classical 
arrows (i.e. arrows colored $i$ for some $i \neq 0$) and Demazure arrows. This could be used to give an algorithm for 
calculating the energy function: given $b$, move backwards along arrows until one reach $u_B$, only using 
Demazure $0$-arrows. Then $D(b) - D(u_B)$ is the number of $0$ arrows in this path. 
\end{remark}
 
\section{Generalizations to other tensor products of KR crystals}
\label{section.nonperfect}
 
Notice that Lemma~\ref{DD-on-B} holds for more general tensor products of KR crystals,
not just tensor products of level-$\ell$ perfect KR crystals. We now show how Theorem~\ref{grading=energy}
and Corollary~\ref{Edeg} can be generalized as well.  
 
For this section, fix $\g$ of non-exceptional affine type, $\ell >0$, and a tensor product
$B= B^{r_N, s_N} \otimes \cdots \otimes  B^{r_1, s_1}$ of KR crystals, such that $\ell \geq \lceil s_k/c_k \rceil$
for all $1 \leq k \leq N$. We call such a crystal a {\it composite KR crystal of level bounded by $\ell$}.

The proof of the following proposition is similar to~\cite[Proof of Theorem 4.4.1]{KMN1:1992}.
 
\begin{prop} \label{is-hw}
For $B$ a composite KR crystal of level bounded by $\ell$,
$$B \otimes B(\ell \Lambda_0) \cong \bigoplus_{\Lambda'} B(\Lambda'),$$
where the sum is over a finite collection of (not necessarily distinct) $\Lambda' \in P^+_\ell$.
\end{prop}
 
\begin{proof}
Fix $b \otimes c \in B \otimes B(\ell \Lambda_0)$. If $c \neq u_{\ell \Lambda_0}$, then for some $i \in I$ we have
$e_i(c) \neq 0$. This implies that, for some $k \geq 1$,
\begin{equation}
        e_i^k(b \otimes c) = e_i^{k-1}(b) \otimes e_i(c) \neq 0.
\end{equation}
In this way, one can apply raising operators to $b \otimes c$ to obtain an element $b' \otimes u_{\ell \Lambda_0}$.
 
Set $M=\dim(B)$, and assume that, for some choice of $e_{i_1}, \ldots, e_{i_M}$, and some
$b \in B$, $ e_{i_M} \cdots e_{i_1}(b \otimes u_{\ell \Lambda_0}) \neq 0$. Then some element 
$b' \otimes u_{\ell \Lambda_0}$ must appear twice. Thus for some $1 \leq r \leq s \leq M$, we have 
$e_{i_s} \cdots e_{i_r}(b' \otimes u_{\ell \Lambda_0})=b' \otimes u_{\ell \Lambda_0}$.
By weight considerations, at least one $e_0$ must appear in this sequence. Furthermore, whenever $e_{i_k}=e_0$, 
we must have $\varepsilon_0(e_{i_{k-1}} \cdots e_{i_r} (b')) > \ell$. Thus  Lemma~\ref{DD-on-B} implies that 
$D(b') < D(b')$, which is a contradiction. 
Hence the left hand side does decompose as a finite direct sum of $B(\Lambda')$ for $\Lambda' \in P^+$. It follows 
by weight considerations that all $\Lambda'$ are in $P^+_\ell$.
\end{proof}
 
\begin{definition} \label{def:gen}
For each $b \in B$, let $u_b^{\ell \Lambda_0}$ be the unique element of $B$
such that $u^{\ell \Lambda_0}_b \otimes u_{\ell \Lambda_0}$ is the highest weight in the component from
Proposition~\ref{is-hw} containing $b \otimes u_{\ell \Lambda_0}$.
\end{definition}
 
Define the function $\deg$ on a direct sum of highest weight crystals to be the basic grading on each component,
with all highest weight elements placed in degree $0$.
 
\begin{corollary} \label{cor:gen1}
Choose an isomorphism $m: B \otimes B( \ell \Lambda_0) \cong \bigoplus_{\Lambda'} B(\Lambda')$.
Then for all $b \in B$, we have $D(b)-D(u_b^{\ell \Lambda_0})= \deg(m(b \otimes u_{ \ell \Lambda_0}))$.
\end{corollary}
 
\begin{proof}
This follows from Lemma \ref{DD-on-B}, just as in the proof of Theorem \ref{grading=energy}.  
\end{proof}
 
The following should be interpreted as a generalization of Corollary~\ref{Edeg}.
 
\begin{corollary} \label{cor:gen2}
The minimal number of $e_0$ in a string of $e_i$ taking $b$ to $u_b^{\ell \Lambda_0}$ is 
$D(b)-D(u_b^{\ell \Lambda_0})$.
\end{corollary}
 
\begin{proof}
This follows as in the proof of Corollary~\ref{Edeg}, using the fact that, for any $b \in B$,
$b \otimes u_{\ell \Lambda_0}$ is connected to $u_b^{\ell \Lambda_0} \otimes u_{\ell \Lambda_0}$ in the affine 
highest weight crystal.
\end{proof}

\begin{remark} \label{general-demazure}
The statements in this section do not give a relationship with Demazure crystals. However, 
Naoi~\cite[Proposition after Theorem B, Proposition 7.6]{Naoi:2010} and~\cite{Naoi:2011}
has recently proven that  $B \otimes u_\Lambda$ is isomorphic to a disjoint union of Demazure crystals inside 
the various irreducible components of $B \otimes B(\Lambda)$, 
where $\Lambda$ is an arbitrary dominant weight of level $\ell \geq 1$ 
and $B$ is a tensor product of perfect KR crystals each of which has level at most $\ell$. 
It would be interesting to determine whether or not this continues to hold for $B$ nonperfect.
\end{remark}
  
\section{Applications}
\label{section.applications}
 
In this section we discuss how the relation between the affine grading in the Demazure crystal
and the energy function can be used to derive a formula for the Demazure character using the
energy function, as well as showing how they are related to nonsymmetric Macdonald
polynomials and Whittaker functions.
 
\subsection{Demazure characters}

Since the character of a Demazure module $V_w(\lambda)$ can be expressed in terms of the Demazure
crystal by~\eqref{equation:Kashiwara-Demazure}, we have the following immediate corollary of 
Theorem~\ref{FSS-theorem}.
 
\begin{corollary} \label{Echar}
Let $B= B^{r_N, \ell c_{r_N}} \otimes \cdots \otimes B^{r_1, \ell c_{r_1}}$ be a $U'_q(\g)$-composite
level-$\ell$ KR crystal, $\lambda = -(c_{r_1} \omega_{r_1^*}+ \cdots + c_{r_N} \omega_{r_N^*})$, and
$t_\lambda = v\tau$ as in Theorem~\ref{FSS-theorem}. Then
\begin{equation}
\label{equation:demazure_character}
       \ch V_v(\ell\Lambda_{\tau(0)})
       = e^{\ell\Lambda_0} \sum_{b\in B} e^{\wt(b) - \delta E^{\intrinsic}(b)}
       = e^{\ell\Lambda_0} \sum_{b\in B} e^{\wt^\aff(b)},
\end{equation}
where $\wt^\aff(b) = \wt(b)-\delta E^\intrinsic(b)$ and $\wt(b)$ is the $U_q'(\g)$-weight of $b$.
\end{corollary}
 
\begin{proof}
Recall that by Theorem~\ref{FSS-theorem} we have the isomorphism
\begin{equation*}
        j \left( B_v(\ell \Lambda_{\tau(0)}) \right) = B \otimes u_{\ell \Lambda_0}.
\end{equation*}
Combining this with~\eqref{equation:Kashiwara-Demazure}, we obtain the result.
\end{proof}
 
As in e.g.~\cite{HKOTT:2002,HKOTY:1999},
the one-dimensional configuration sums are defined as
\begin{equation} \label{equation.X}
        X(\mu;B) = \sum_{\substack{b\in B, \wt(b)=\mu\\ \text{$e_i(b)=\emptyset$ for $i\in I\setminus \{0\}$}}}
        q^{-D(b)} \; ,
\end{equation}
where $B = B^{r_N,s_N} \otimes \cdots \otimes B^{r_1,s_1}$. 
By Theorem~\ref{E=D} we can rewrite Corollary~\ref{Echar} as follows (compare also 
with~\cite[Theorem 1.2]{Shim:2002}).

\begin{corollary}
With the same assumptions and notation as in Corollary~\ref{Echar} and setting $q=e^\delta$ we have
\begin{equation}
       \ch V_v(\ell\Lambda_{\tau(0)}) = e^{\ell\Lambda_0} \; q^{D(u_B)} \; \sum_\mu
	X(\mu;B) \; \ch(V(\mu))\; ,
\end{equation}
where $V(\mu)$ is the module of highest weight $\mu$ for the underlying finite type algebra.
\end{corollary}
 
\begin{example} 
Let us illustrate various quantities used in this section.  Consider $B=(B^{1,1})^{\otimes 3}$ of type $A_2^{(1)}$. 
This has $ u_B = 3\otimes 2\otimes 1$. Take $  b = 2\otimes 3\otimes 1$.
Then one can calculate $D(b)=-2$ and  $D(u_B)=-3$, so that $E^\intrinsic(b)=D(b)-D(u_B)=1$. Also
\begin{equation*}
        \wt(b) = (\Lambda_2-\Lambda_1) + (\Lambda_0-\Lambda_2) + (\Lambda_1-\Lambda_0)=0,
\end{equation*}
so that $\wt^\aff(b)=-\delta$.
\end{example}
 
\subsection{Nonsymmetric Macdonald polynomials}
\label{section.nonsymmetric_mac}

Fix $\g$ of affine type. Let $\widetilde P \subset P$ be the sublattice of level $0$ weights. Recall that 
$\widetilde P$ is naturally contained in $\overline{P} + \bz \delta$, where $\delta$
is the null root and we identify $\overline P \otimes_\bz {\Bbb R}$ with a subspace of $P \otimes_\bz {\Bbb R}$ 
by identifying the finite type simple roots with their corresponding affine simple roots (and this containment is 
equality except in type $A_{2n}^{(2)}$). Let $t$ be the collection of indeterminates $t_\alpha$ for each root 
$\alpha$ such that $t_\alpha = t_{\alpha'}$ if $\alpha$ and $\alpha'$ have the same length.
Consider the following elements of the group algebra $\Q(q,t) \overline{P}$:
\begin{equation*}
        \Delta := \prod_{\alpha\in R_+^\aff} \frac{1-e^\alpha}{1-t_\alpha e^\alpha} \big{|}_{e^{\delta}=q}, \quad \text{ and } \quad \Delta_1:=\Delta/([e^0]\Delta), 
\end{equation*}
where $[e^0]$ means the coefficient of $e^0$ and
$R_+^\aff$ is the set of positive affine real roots.
Cherednik's inner product~\cite{Cherednik:1995} on $\Q(q,t) \overline P$ is $\langle f,g\rangle_{q,t} = [e^0](f\overline{g}\Delta_1)$,
where $\overline{\cdot}$ is the involution $\overline{q}=q^{-1}$, $\overline{t}=t^{-1}$,
$\overline{e^{\lambda}}=e^{-\lambda}$.
 
The nonsymmetric Macdonald polynomials $E_\lambda(q,t)\in \Q(q,t)\overline P$ for $\lambda\in \overline P$
were introduced by Opdam~\cite{Opdam:1995} in the differential setting and
Cherednik~\cite{Cherednik:1995} in general (although here we follow conventions of Haglund, Haiman, Loehr
~\cite{Haiman:2006,HHL:2008}).
They are uniquely 
characterized by two conditions:
(Triangularity): $E_\lambda \in x^\lambda + \Q(q,t)\{x^\mu \mid \mu<\lambda\}$ and 
(Orthogonality): $\langle E_\lambda, E_\mu \rangle_{q,t} = 0$ for $\lambda\neq \mu$.
Here $<$ is the Bruhat ordering on $\overline P$ identified 
with the set of minimal coset
representatives in $\widetilde W/\oW$, where $\widetilde W$ is the extended affine Weyl
group and $\oW$ is the classical Weyl group.
  
Extending Sanderson's work~\cite{San:2000} for type $A$, Ion~\cite{Ion:2003} showed that for all simply laced
untwisted affine root systems, the specialization of the nonsymmetric Macdonald polynomial $E_\lambda(q,t)$ 
at $t=0$ coincides with a specialization of a Demazure character of a level one affine integrable module
(see~\cite[Theorem 6]{San:2000}, \cite{Ion:2003}):
Write $t_\lambda \in \widetilde W$ as $t_\lambda= w \tau$, where $w \in W, \tau \in \Sigma$. Then
\begin{equation}\label{macdem}
       E_{\lambda}(q,0)=q^c\, \mathrm{ch}(V_{w}(\Lambda_{\tau(0)}))|_{e^\delta=q, \; e^{\Lambda_0}=1},
\end{equation}
where
$c$ is a specific exponent described in~\cite{Ion:2003, San:2000} (and described in 
Corollary~\ref{Mac=KR} below in types $A_n^{(1)}$ and $D_n^{(1)}$). For $\lambda$ a single row, this 
relation also follows from combining~\cite{KMOTU:1998,Hikami:1997}. 

\begin{remark}
We have used the conventions of Haglund, Haiman, Loehr~\cite{Haiman:2006,HHL:2008}, which differ from
those in~\cite{Ion:2003} by the change of variables $q \rightarrow q^{-1}$.  
\end{remark}

If $\lambda$ is anti-dominant and $\g= A_n^{(1)}$ or $D_n^{(1)}$, we can apply Theorem~\ref{FSS-theorem} and 
Corollary~\ref{Edeg} to give a connection with KR crystals, along with their
energy. 

\begin{corollary} \label{Mac=KR} Fix $\g = A_n^{(1)}$ or $D_n^{(1)}$. 
Fix an anti-dominant weight $\lambda$, and write $\lambda = -(\omega_{r_1^*} + \cdots + \omega_{r_N^*})$.
Let $B = B^{r_N,1} \otimes \cdots \otimes B^{r_1,1}$. Then 
\begin{equation*}
 P_\lambda(q,0)= E_\lambda(q,0) =\sum_{b \in B} q^{-D(b)}e^{\wt(b)}. 
\end{equation*}
\end{corollary}

\begin{proof}
The equality $P_\lambda(q,0)= E_\lambda(q,0) $ is \cite[Theorem 4.2]{Ion:2003}. 
By Theorem~\ref{FSS-theorem} and Corollary~\ref{cor:gen1} (noting that $c_r=1$ for types $A_n^{(1)}$ and $D_n^{(1)}$
as in Figure~\ref{u-table}), equation~\eqref{macdem} implies
\begin{equation*}
    P_\lambda(q,0)= q^c \sum_{b \in B} q^{-D(b)+D(u_B)}e^{\wt(b)}. 
\end{equation*}
It remains to show that $c= - D(u_B)$.
Note that the generator $v=v_{r_N,1} \otimes \cdots \otimes v_{r_1,1} \in B$ has weight $\wt(v)=-\lambda^*=w_0(\lambda)$. 
Since $P_\lambda(q,0)$ is symmetric, the coefficient of $e^{w_0(\lambda)}$ is the same as the coefficient of $e^\lambda$, 
which by definition of the Macdonald polynomials, is 1. Hence we obtain the condition
$c+D(u_B)-D(v)=0$ which implies by our normalization $D(v)=0$ that $c=-D(u_B)$.
\end{proof}

The restriction to types $A_n^{(1)}$ and $D_n^{(1)}$ in Corollary~\ref{Mac=KR}
stems from the use of Ion's result~\cite{Ion:2003}. An extension of Corollary~\ref{Mac=KR}
to type $C_n^{(1)}$ follows from~\cite{Lenart:2012,LS:2011}. It would be interesting to determine if this 
result continues to hold in other types. 

\begin{remark}
In order to  match our notation for nonsymmetric Macdonald polynomials, we index symmetric Macdonald 
polynomials by anti-dominant weights. Many people (see for example~\cite{Knop:1997,Marshall:1999})
index Macdonald polynomials by dominant weights. The correct conversion between these conventions is that, 
for a dominant integral weight $\lambda$, $P_{\lambda}$ in the above references is denoted by $P_{w_0(\lambda)}$ here. 
\end{remark}

Corollary~\ref{Mac=KR} implies that the coefficients in the expansion of the symmetric Macdonald polynomial 
$P_\lambda(q,0)$ at $t=0$ in terms of the irreducible characters $\ch(V(\mu))$ coincide with 
$X(\mu;B^{r_N,c_{r_N}} \otimes \cdots \otimes B^{r_1,c_{r_1}})$ of~\eqref{equation.X}. In formulas
\begin{equation*}
        P_\lambda(q,0) = \sum_\mu X(\mu;B) \; \ch(V(\mu)),
\end{equation*}
where $B=B^{r_N,c_{r_N}} \otimes \cdots \otimes B^{r_1,c_{r_1}}$ with the $r_i$ determined from $\lambda$ as in
Theorem~\ref{FSS-theorem}.
 
\begin{remark}
In the case $\g = A_n^{(1)}$, one can define nonsymmetric Macdonald polynomials $E_\lambda(q,t)$ for
any $\mathfrak{gl}_n$ weight $\lambda$ (although if $\lambda$ and $\mu$ correspond to the same
$\mathfrak{sl}_n$ weight, these only differ by a scalar). Letting $\overline P'$ denote the lattice of
$\mathfrak{gl}_n$ weights, there is a natural injection
\begin{equation*}
\begin{aligned}
\Q(q,t) \overline P' &  \rightarrow \Q(q,t)[x_1^{\pm 1}, \ldots, x_n^{\pm 1} ] \\
e^\lambda & \mapsto x_1^{\lambda_1} x_2^{\lambda_2} \cdots x_n^{\lambda_n},
\end{aligned}
\end{equation*}
so we can identify $E_\lambda(q,t)$ with an actual polynomial, and we do so in the examples below
by writing $E_\lambda(x;q,t)$.
\end{remark}
 
\begin{example} \label{mmm}
The Macdonald polynomial of type $A_2^{(1)}$ indexed by the anti-dominant weight $(0,0,2)$ is given by
\begin{equation*}
P_{(0,0,2)}(x;q,0) = x_1^2 + (q+1) x_1 x_2 + x_2^2 + (q+1) x_1 x_3 + (q+1) x_2 x_3 + x_3^2.
\end{equation*}
By Corollary~\ref{Mac=KR}, this is given by the character of $B^{2^*, 1} \otimes B^{2^*, 1}$, 
where the power of $q$ counts $E^\intrinsic$. Here $2^*=1$, so we consider the KR crystal $B^{1,1} \otimes B^{1,1}$. 
Under the map from Theorem~\ref{FSS-theorem}, this is isomorphic as a classical crystal to 
$B_{s_2s_1s_0s_2}(\Lambda_2)$, and there 
are enough $0$ arrows in this Demazure crystal to calculate $E^\intrinsic$. The KR crystal is then given as follows,
where to make the picture cleaner we have only included arrows which survive in the Demazure crystal. Note that 
the ground state path is $2 \otimes 1$.
\begin{equation} \label{equation.002_example}
\begin{array}{lll}
       2\otimes 1 \stackrel{2}{\longrightarrow} 3 \otimes 1&\stackrel{0}{\longrightarrow} 1 \otimes 1
       \stackrel{1}{\longrightarrow} 1 \otimes 2 &\stackrel{1}{\longrightarrow} 2 \otimes 2
       \stackrel{2}{\longrightarrow} 2 \otimes 3 \stackrel{2}{\longrightarrow} 3 \otimes 3\\
       & \stackrel{1}{\searrow} \quad 3\otimes 2 \hspace{0.3cm}
       & \stackrel{2}{\searrow} \quad 1 \otimes 3 \quad \color{red}{\stackrel{1}{\nearrow}}
\end{array}
\end{equation}
The black arrows are all arrows that appear in $f_2^{n_1} f_1^{n_2} f_0^{n_3} f_2^{n_4} (2 \otimes 1)$ for
some exponents $n_i\ge 0$, which give all vertices of the Demazure crystal, but not all arrows. The red arrows are 
the additional arrows in the Demazure crystal. 
We calculate that $D(2 \otimes 1)= D(3 \otimes 1)=D(3 \otimes 2)=-1$, and the rest of the elements of 
$B^{1,1} \otimes B^{1,1}$ have $D=0$. This confirms Corollary~\ref{Mac=KR} in this case. 
\end{example}
 
The connection with KR crystals can also be used to compute all nonsymmetric Macdonald polynomials specialized at $t=0$. 
Fix $\lambda$ which need not be anti-dominant, and define $w, \tau$ as in \eqref{macdem}. There is a unique minimal 
$w' \in W_0$ such that $w'( \lambda)$ is an anti-dominant weight. Then $w' w ((w')^{-1})^\tau \tau$ is a translation by an 
anti-dominant weight, where the superscript $\tau$ means conjugation by $\tau$. Also, $B_{w}(\Lambda_{\tau(0)})$ embeds in 
$B_{w'w ((w')^{-1})^\tau}(\Lambda_{\tau(0)}) = B_{w'w}(\Lambda_{\tau(0)})$, where the final equality follows because $w' \in W_0$. 
Let $j: B_{w'w}(\Lambda_{\tau(0)}) \rightarrow B$ be the isomorphism from Theorem~\ref{FSS-theorem}, where $B$ is the 
appropriate composite KR crystal, and let $B' = j(B_{w}(\Lambda_{\tau(0)}))$. Then  in types $A_n^{(1)}$ and $D_n^{(1)}$ we have 
\begin{equation} \label{non-sym-KR}
 E_\lambda(q,0)= \sum_{b \in B'} q^{-D(b)}e^{\wt(b)} \;,
\end{equation}
where $D$ is calculated in the ambient composite KR crystal $B$. The subset $B'$ of $B$ can be described as
$
    \{ f_{i_k}^{n_k} \cdots f_{i_1}^{n_1} u_B \mid n_1, \ldots, n_k \geq 0  \},
$
where $s_{i_k} \cdots s_{i_1}$ is any reduced expression for $w$.
 
\begin{example} Again consider type $A_2^{(1)}$, and take the nonsymmetric Macdonald polynomial
\begin{equation*}
        E_{(0,2,0)}(x;q,0) = x_1^2 + (q+1) x_1 x_2 + x_2^2 + q x_1 x_3 + q x_2 x_3.
\end{equation*}
Then $w, \tau$ from~\eqref{macdem} are $w=s_1s_0s_2s_1$ and $\tau= 0 \rightarrow 2 \rightarrow 1 \rightarrow 0$. 
The shortest $w' \in W_0$ so that $w' (0,2,0)$ is anti-dominant is $w'=s_2$. Thus the above argument along with 
Example~\ref{mmm} shows that $B_w(\Lambda_{\tau(0)}) = B_{s_1s_0s_2s_1} (\Lambda_2) = B_{s_1 s_0 s_2} (\Lambda_2)$ embeds into $B= B^{1,1} \otimes B^{1,1}$ 
(as a classical crystal), and a simple verification shows that the image $B'$ is everything but
$1 \otimes 3, 2 \otimes 3$ and $3 \otimes 3$. This indeed verifies that $E_{(0,2,0)}(x;q,0)$ is given 
by~\eqref{non-sym-KR}. 
 \end{example} 

\begin{example} 
Consider the Macdonald polynomial of type $A_2^{(1)}$
\begin{equation*}
P_{(0,1,2)}(x;q,0) = x_1^2 x_2 + x_1 x_2^2 + x_1^2 x_3 + (q+2) x_1 x_2 x_3+ x_2^2 x_3 + x_1 x_3^2
+ x_2 x_3^2.
\end{equation*}
Our results say that this is given by the character of the affine crystal $B$, where $B= B^{2,1} \otimes B^{1,1}$. 
The ground state path is $\begin{array}{c}3\\2\end{array} \otimes 1$.
As a classical crystal, $B$ consists of two components, with highest weight elements 
$\begin{array}{c}3\\2\end{array} \otimes 1$ and $\begin{array}{c}2\\1\end{array} \otimes 1$ of weight $0$ and 
$\omega_1 + \omega_2$ respectively. One can check that $D(\begin{array}{c}3\\2\end{array} \otimes 1) =-1,$ 
and $D(\begin{array}{c}2\\1\end{array} \otimes 1)=0.$ This indeed confirms
\begin{equation*}
        P_{(0,1,2)}(x,q,0) = q + \ch(V(\omega_1+\omega_2)). 
\end{equation*}
\end{example}
 
\subsection{$q$-deformed Whittaker functions}
Gerasimov, Lebedev, Oblezin~\cite[Theorem 3.2]{GLO:2008} showed that $q$-deformed
$\mathfrak{gl}_n$-Whittaker functions are Macdonald polynomials specialized at $t=0$. As above this also gives 
a link to Demazure characters, and hence by the results in Section \ref{section.energy.equiv} to KR crystals graded 
by their energy functions. It would be interesting to generalize this to other types (in particular type $D_n^{(1)}$, where both Ion's results \cite{Ion:2003} and our results from Section \ref{section.energy.equiv} hold). 
The $q$-deformed $\mathfrak{gl}_n$-Whittaker functions are simultaneous eigenfunctions
of a $q$-deformed Toda chain, which might serve as a starting point for this generalization.
See also the recent paper by Braverman and Finkelberg~\cite{BF:2012}.
Brubaker, Bump and Licata~\cite{BBL:2011} constructed a natural basis of the Iwahori fixed vectors in the 
Whittaker model using Demazure-Lusztig operators. This gives another avenue to the link between
Demazure characters and Whittaker functions.
 

\end{document}